\documentclass[12pt,a4paper]{article}
\usepackage{amsmath,amssymb,amsthm,graphicx,color}
\usepackage[left=1in,right=1in,top=1in,bottom=1in]{geometry}
\usepackage{cite}
\usepackage[british]{babel}

\newtheorem{theo}{Theorem}[section]
\newtheorem{lm}{Lemma}[section]

\newtheorem{rmk}{Remark}[section]
\newtheorem{proposition}{Proposition}[section]
 
\allowdisplaybreaks

\numberwithin{equation}{section}

\def\eps{\varepsilon}
\def\R{{\mathbb R}}
\def\Z{{\mathbb Z}}
\def\N{{\mathbb N}}

\def\I{{\mathcal{I}}}
\def\Exp{\mathbb{E}}
\def\Pr{\mathbb{P}}
\def\1{{\mathbf 1}}

\def\eps{\varepsilon}

\def\re{{\mathrm e}}

\def\0{{\mathbf 0}}

\def\po{\preccurlyeq}

\newcommand{\ud}{{\mathrm d}}
\newcommand{\ord}{{\mathop{\mathrm{ord}}}}

\newcommand{\tod}{\stackrel{{d}}{\rightarrow}}

\title{Rank-driven Markov processes}

\author{Michael Grinfeld \qquad Philip A. Knight \qquad Andrew R. Wade \\
\normalsize{University of Strathclyde}}

\begin{document}

\maketitle

\begin{abstract}
We study a class of Markovian systems of $N$ elements taking values in $[0,1]$ that evolve
in discrete time $t$
via randomized replacement rules based on the ranks of the elements. These rank-driven processes
 are inspired
by variants of the Bak--Sneppen model of evolution, in which the system represents an evolutionary
`fitness landscape' and which is famous as a simple
model displaying self-organized criticality.
Our main results are concerned with long-time
large-$N$ asymptotics for
 the general model in which, at each time step,
$K$ randomly chosen elements are discarded and replaced by independent
$U[0,1]$ variables, where the ranks of the elements to be replaced are chosen, independently at each time step,
according to a  distribution $\kappa_N$ on $\{1,2,\ldots,N\}^K$. 
Our main results are that, under   appropriate conditions on $\kappa_N$, 
the system exhibits threshold behaviour at $s^* \in [0,1]$, where $s^*$ is a
function of $\kappa_N$, and
the marginal distribution of a randomly selected element converges to $U[s^*, 1]$ as
$t \to \infty$ and $N \to \infty$. 
Of this class of models, results in the literature have previously been given for 
special cases only, namely the `mean-field' or `random neighbour' Bak--Sneppen model. Our proofs 
avoid the heuristic arguments of some of the previous work and   use   Foster--Lyapunov ideas. Our results extend existing results
and establish their natural, more general context. We derive some more specialized results for the particular case where $K=2$.
One of our  technical tools is a result on convergence of stationary distributions
for families of uniformly ergodic Markov chains on increasing state-spaces, which may be of independent interest.
\end{abstract}



\section{Introduction}

Bak and Sneppen \cite{bs} introduced a simple stochastic model of 
evolution which initiated a considerable 
body of research by
physicists and mathematicians. The   Bak--Sneppen model
has proved so influential
because it is simple to describe and not difficult to simulate, and,
while being  challenging to analyse
rigorously,
demonstrates highly non-trivial behaviour: it is said to exhibit
`self-organized criticality' (see e.g.\ \cite{jensen}).

The Bak--Sneppen model is as follows.
Consider the sites $1,2, \ldots, N$ arranged cyclically,
so that site $k$ has neighbours $k-1$ and $k+1$ (working modulo $N$).
Each site, corresponding to a species in the model,
is initially assigned an independent $U[0,1]$ random variable
representing a `fitness' value for the species; here and subsequently $U[a,b]$ stands for the uniform
distribution on the interval $[a,b]$.
The Bak--Sneppen model
is a discrete-time Markov process, where
 at each step the minimal
fitness value   and the values at the two neighbouring sites
are replaced
by three independent $U[0,1]$ random variables.
A variation on the model is the (maximal) {\em anisotropic} Bak--Sneppen model \cite{hd} in which, at each step,
only the {\em right} neighbour of the site with minimal fitness is updated along with the minimal value.

A large physics literature is devoted to these models. Simulations suggest that the
equilibrium distribution of the fitness at any particular site approaches  $U[s^*, 1]$
 in the $N \to \infty$ limit, for some threshold value $s^*$; simulations
give $s^* \approx 0.667$ for the original Bak--Sneppen model and $s^* \approx 0.724$
for the anisotropic model \cite{gd}.
There is a much smaller number of mathematical papers on the model and its variants: see e.g.\ \cite{gmv,gmn,mz2,mz3};
see also the thesis \cite{gillett}. It is a challenge
to obtain further rigorous results for such models.

A simpler model can be formulated by removing the underlying topology, and such `mean field'
or `random neighbour'
versions of the model have also received attention in the literature: see e.g.\ \cite{deboer,deboer2,fbs,jensen,lp,pismak}.
For example,  the mean-field version of the
anisotropic Bak--Sneppen model 
 again has $N$ sites each endowed
with a fitness in $[0,1]$. At each step, the minimal fitness is replaced, along with one other
fitness chosen {\em uniformly at random} from the remaining $N-1$ sites. Again the replacement fitnesses
are independent $U[0,1]$ variables.

Such mean-field models display some features qualitatively similar to the original Bak--Sneppen model,
but give little indication of how the distinctive asymptotics of the Bak--Sneppen model,
in which the topology plays a key role, might arise. In particular, changing the topology of the model
  changes the value of the threshold $s^*$ in a way that the mean-field models cannot account for.
In the present paper we study some generalizations of the mean-field model described informally above,
which we call {\em rank-driven processes}. These models  
 represent one possibility for showing how the influence of topology might be replicated by simpler features.

In these more general models, we again have $N$ sites, and at each time
step some fixed number of fitness values are selected for replacement, but for this selection process
we allow  general 
stochastic rules based on the {\em ranks} of the values. These rank-driven processes
are Markov processes on {\em ranked} sequences, or {\em order statistics} (see Sections \ref{warmup}
and \ref{model} for formal definitions).

To give a concrete example,  we could, at each step, replace the minimal fitness along with the $R$th ranked value,
where $R$ is chosen independently at each step from some distribution on $\{2,\ldots,N\}$, with $R=2$ corresponding
to the second smallest value, and so on. This model generalizes that of \cite{deboer}, which has $R$
{\em uniform} on $\{2,\ldots,N\}$, and exhibits much richer behaviour. Specifically, the threshold $s^*$ depends explicitly
on the distribution chosen for $R$: in this way, the distribution of $R$ is playing a role analogous to the
underlying topology in the Bak--Sneppen models.

Such rank-driven processes are of interest in their own right, but our motivation
for studying them also arises from an attempt to understand the original Bak--Sneppen model,
where the topology plays a key role. While the Bak--Sneppen model
can be viewed as a Markov process on the space $[0,1]^N$, it
gives rise to a decidedly non-Markovian process on   order statistics. At the same time,
as we explore in detail in \cite{gkw1}, there is an algorithmic way to associate to the Bak--Sneppen
model a rank-driven process (a process that {\em is} Markovian on order statistics) and 
which, according to numerical evidence, shares a number of important properties with the Bak--Sneppen
model. The aim of the present paper is to present a rigorous analysis of rank-driven processes.
 
The outline of the remainder of the paper is as follows. In Section \ref{warmup}
we discuss an introductory example in which a single value is updated at each step.
In Section \ref{model} we describe the general rank-driven processes
that we consider and state our main theorems on asymptotic behaviour.
In Section \ref{main} we focus on a specific class of examples, generalizing the
mean-field anisotropic Bak--Sneppen model \cite{deboer,fbs}, and give some more detailed results.
We emphasize the difference in nature of the results in Sections
\ref{model} and \ref{main}: in the former,
the results cover a very general class of processes
and the proofs are   robust, using Foster--Lyapunov arguments
and general theory of Markov processes, while in the latter,
we specialize to a narrower class of processes and
exploit their special structure. It is likely that analogues of our
results from Section \ref{main} could be obtained for other
processes, but the details would depend on the particular processes studied.
In Section \ref{open} we make some further remarks and state some open
problems. Section \ref{sec:genproofs} is devoted to the proofs of the main
results in Section \ref{model}, while Section \ref{sec:proofs2} is devoted
to the proofs of the results in Section \ref{main}. The Appendix, Section \ref{sec:limits},
gives one of our key technical tools  on the asymptotics of
families of Markov chains that are uniformly ergodic in a precise sense.

  \section{Warm-up example: Replace the $k$th-ranked value}
  \label{warmup}

We start by describing a particularly simple model, which can be solved completely,
to demonstrate some ideas that will be useful in greater generality later on.
It will be convenient to view all of our models as Markov chains on {\em ranked} sequences, or {\em order statistics}.
Fix $N \in \N := \{1,2,\ldots\}$.
Given a vector $(x^1, x^2, \ldots, x^N)$ we write the corresponding increasing   order statistics  as
\[ (x^{(1)}, x^{(2)}, \ldots, x^{(N)} ) = \ord ( x^1, \ldots, x^N ), \]
where $x^{(1)} \leq x^{(2)} \leq \cdots \leq x^{(N)}$.
Let $\Delta_N$ denote the `simplex'
\[ \Delta_N := \{ (x^1,\ldots, x^N) \in [0,1]^N : x^1 \leq x^2 \leq \cdots \leq x^N \}.\]
We study 
 stochastic processes on $\Delta_N$ indexed by discrete time $\Z^+ := \{0,1,2,\ldots\}$.
 
Let $U_1, U_2, \ldots$ denote a sequence of independent $U[0,1]$ random variables.
Define a Markov process $X_t$ on $\Delta_N$
with transition rule such that, given $X_t$, 
\[ X_{t+1 } = {\rm ord} \{ U_{t+1} , X_t^{(2)}, X_t^{(3)}, \ldots, X_t^{(N)} \} ;\]
in other words, at each step,
 discard the {\em smallest} value  and replace it by a $U[0,1]$
random variable. To make clear the dependence on the model parameter
$N$, we  write  $\Pr_N$ for the probability measure associated with this model
and $\Exp_N$ for the corresponding expectation.

It is natural to anticipate that $X_t$ should approach (as $t \to \infty$)
a limiting
(stationary) distribution; we show in this section that this is indeed the case.
Assuming such a stationary distribution exists, and is unique, we can guess what it must be:
the distribution of the random vector $(U, 1, 1, 1, \ldots, 1)$ (a
 $U[0,1]$ variable followed by $N-1$ units) is invariant
under the evolution of the Markov chain. The process $X_t$ itself lives on a relatively complicated state-space,
and at first sight it might seem that some fairly sophisticated argument would be needed to show
that it has a unique stationary distribution. In fact, we can reduce the problem to a  simpler
problem on a finite
state-space as follows.

For each $s \in [0,1]$, define the {\em counting function}
\begin{equation}
\label{Ndef}  C^N_t (s) := \# \{ i \in \{1 ,2 , \ldots, N \} : X_t^{(i)} \leq s \}
= \sum_{i=1}^N \1 \{ X_t^{(i)} \leq s \} ,\end{equation}
 i.e., $C^N_t(s)$ is the number of values
of magnitude at most $s$ in the system at time $t$. (Here and throughout we use $\# A$ to denote the
number of elements of a finite set $A$.)
Then, for a fixed $t$,
 $X_t$ is characterized by the counting functions $(C^N_t(s))_{s \in [0,1]}$.
For a specific $s$, $C^N_t(s)$ encodes marginal information about the
$X_t^{(k)}$, since the two events $\{ C^N_t (s) \geq k \}$ and $\{ X_t^{(k)} \leq s \}$ are equivalent.
By an analysis of the auxiliary stochastic processes $C^N_t(s)$ we will prove the following result,
which deals with the more general model in which, at each time step, the  point with rank $k$ is replaced.

\begin{proposition}
\label{smallest}
Let $N \in \N$ and $k \in \{1,2,\ldots, N\}$.
For the model in which at each step we replace the $k$th-ranked value by an independent $U[0,1]$
value, we have that, as $t \to \infty$,
\[ (X_t^{(1)}, X_t^{(2)}, \ldots, X_t^{(N)} ) \tod (0, \ldots, 0, U, 1,   \ldots, 1 ) ,\]
where $U$, the $k$th coordinate of the limit vector, has a $U[0,1]$ distribution.
\end{proposition}

If $k = k(N)$ is such that $k(N)/N \to \theta \in [0,1]$ as $N \to \infty$,
a consequence of Proposition \ref{smallest} is that the distribution
of a uniformly chosen point converges (as $t \to \infty$ and then $N \to \infty$)
to the distribution with an atom of mass $\theta$ at $0$ and an atom of mass $1-\theta$ at $1$. For example,
if we always replace a {\em median} value, $\theta = 1/2$ and the limit distribution
has two atoms of size $1/2$ at $0$ and $1$. 

\begin{proof}[Proof of Proposition \ref{smallest}.]
It is not hard to see that $C^N_t(s)$ is a Markov chain on $\{0,1,2,\ldots, N\}$. 
The transition probabilities $p^s_N (n,m) := \Pr_N [ C^N_{t+1}(s) = m \mid C^N_t(s) = n]$ are given
for $n \in \{0,\ldots, k-1\}$ by 
$p^s_N(n,n) = 1-s$ and $p^s_N(n,n+1) = s$, and for 
$n \in \{k,\ldots,N\}$ by
$p^s_N(n,n)= s$ and $p^s_N(n,n-1) = 1-s$.
For $s  \in (0,1)$ the Markov chain is reducible and has a single recurrent class consisting of the states $k-1$ and $k$.
It is easy to compute the stationary distribution and  for $s \in (0,1)$ we obtain, 
\[ \lim_{t \to \infty} \Pr_N [ C^N_t(s) = n ] = \begin{cases} 1 -s & \textrm{ if } n =k-1\\
s & \textrm{ if } n=k\\
0 & \textrm{ if } n \notin \{k-1,k\} \end{cases}, \]
by standard Markov chain
 theory. 
In particular, for $s \in (0,1)$,
\[ \lim_{t \to \infty} \Pr_N [ X_t^{(k)} \leq s ] = \lim_{t \to \infty} \Pr_N [ C^N_t(s) \geq k ] = s .\]
That is, $X_t^{(k)}$ converges in distribution to a $U[0,1]$ variable.
Moreover, if $k > 1$, for any $s \in (0,1)$,
$\Pr[ X_t^{(k-1)} \leq s] = \Pr [ C_t^N(s) \geq k-1] \to 1$, which implies that
$X_t^{(k-1)}$ converges in probability to zero. Similarly, if $k < N$, 
for any $s \in (0,1)$,
$\Pr[ X_t^{(k+1)} \leq s] = \Pr [ C_t^N(s) \geq k+1] \to 0$, which implies that
$X_t^{(k+1)}$ converges in probability to 1.
Thus we have proved the marginal result that, as $t \to \infty$, for $U$ a $U[0,1]$ random variable,
\[ X_t^{(i)} \to 0 , ~ (i < k), ~~~ X_t^{(k)} \to U, ~~~  X_t^{(i)} \to 1, ~(i > k) ,\]
 in distribution.
  Then the Cram\'er--Wold device (convergence in distribution
  of an $N$-dimensional random vector is implied by convergence in distribution
  of all linear combinations of its components: see
  e.g.\ \cite[p.\ 147]{durrett})
  together with Slutsky's theorem (if $Y_n$ converges
  in distribution to a random limit $Y$ and $Z_n$ converges
  in distribution to a deterministic limit $z$, then
  $Y_n + Z_n$ converges in distribution to
  $Y+z$: see e.g.\ \cite[p.\ 72]{durrett})
  enable us to deduce the joint convergence.
\end{proof}

\begin{rmk} This simple example shows special features that will not recur in the general case.
(i) Here we obtained a result for any fixed $N \geq 1$;
in the general case, we will typically state results as $N \to \infty$.
(ii) Since all but one of the $X^{(i)}_t$ had a degenerate   limit distribution, we were
able to use a soft argument to deduce convergence of the joint distribution
of $(X^{(1)}_t, \ldots, X^{(N)}_t)$ from the convergence of the marginal
distributions. 
\end{rmk}

\section{Rank-driven processes and threshold behaviour}
\label{model}

In this section we give a general definition of a {\em rank-driven process}
and present some fundamental results on its asymptotic properties.
Fix $N$ (the number of points) and $K \in \{1,\ldots, N\}$
(the  number of replacements at each step).
Define the set
\[ \I^K_N := 
\{ 1,2,\ldots, N \}^K  .\]
The model will be specified by a {\em selection distribution}.
Let $R^N$ denote a random $K$-vector with distinct components in $\{1,\ldots,N\}$.
In components, write $R^N = (R^N_1,\ldots,R^N_K)$. We suppose that $R^N$ is {\em exchangeable}, i.e., its distribution is invariant
under permutations of its components. 
The distribution of $R^N$ can be described by a 
probability mass function $\kappa_N : \I^K_N  \to [0,1]$ that is symmetric under permutations of its arguments,
so $\Pr_N[ R_1^N = i_1, \ldots, R_K^N = i_K] = \kappa_N (i_1, \ldots, i_K)$.

We  define a Markov chain $(X_t)_{t \in \Z^+}$ of
ranked sequences $X_t = (X_t^{(1)}, \ldots, X_t^{(N)})$.
The initial distribution $X_0$ can be arbitrary.
The randomness of the process will be introduced via
independent $U[0,1]$ random variables  $U_1, U_2, \ldots$ 
and independent copies of $R^N$, which we denote by
$R^N(1), R^N(2), \ldots$. In components, write
$R^N(t) = (R^N_1(t), \ldots, R^N_K(t))$. The transition law of the Markov
chain is as follows.

Given $X_t$, discard the elements of (distinct) ranks specified
by $R^N_1(t+1), \ldots, R^N_K(t+1)$ and replace
them by $K$ new independent $U[0,1]$ random variables, namely $U_{Kt+1}, \ldots U_{Kt+K}$;
 then rank the new sequence.
That is, we take $X_{t+1}$ to be
\[ \ord \left(X_t^{(1)}, X_t^{(2)}, \ldots, X_t^{(R^N_1(t+1) -1)},
U_{Kt+1}, X_t^{(R^N_2(t+1)+1)}, \ldots, U_{Kt+K}, X_t^{(R^N_K(t+1)+1)} , \ldots, X_t^{(N)} \right) .\]
Note that ties are permitted.

For $i \in \N$ let $g_N(i) := \Pr_N [ R_1^N = i]$,
the marginal distribution of a specific component of $R^N$. Denote the corresponding
distribution function by
\begin{equation}
\label{GNdef}
 G_N(n) := \Pr_N [ R_1^N \leq n] = \sum_{i=1}^n g_N (i) = \sum_{i_1=1}^n \sum_{i_2=1}^N \cdots \sum_{i_K=1}^N \kappa_N (i_1,i_2, \ldots, i_K).\end{equation}
We make some further assumptions on the selection distribution. Assumption (A1) will ensure
that an irreducibility property holds,
excluding some degenerate cases, while (A2) regulates the $N \to \infty$ behaviour of the selection rule.
\begin{itemize}
\item[(A1)] If $K=1$, suppose that $g_N (i) >0$ for 
 all $i \in \{1,\ldots,N\}$. If $K \geq 2$, suppose that $g_N(1) >0$.
\item[(A2)] Suppose that for any $k \leq K$, 
for all distinct $i_1, \ldots, i_k \in \N$, the limit
$\kappa(i_1,\ldots, i_k) := \lim_{N \to \infty} \Pr_N [ R^N_1 = i_1,\ldots, R^N_k =i_k ]$ 
exists.
\end{itemize}

Note that the limits in (A2) 
need not constitute proper distributions on $\N^k$: there
may be some loss of mass. Indeed, the possibility of a defective distribution
as the limit in (A2) plays a central role in the asymptotics of the rank-driven process,
as we shall describe below. A proper distribution
can be recovered on $(\N \cup \{ \infty \})^K$ by correctly accounting for the lost mass,
and then (A2) can be interpreted as saying that $R^N$ converges in distribution
to a random vector on $(\N \cup \{ \infty \})^K$: see Section \ref{sec:genlimitchain} for details.
 A consequence of (A2) is that
\begin{equation}
\label{Gdef}
\lim_{N \to \infty} g_N (n) = g(n)  \textrm{  and  } \lim_{N \to \infty} G_N (n) = G(n)
\end{equation}
exist  for all $n \in \N$;
then 
 $G$ is a (possibly defective) distribution function on $\N$. (Note that $g(i) = \kappa (i)$.)
  Given (A2), we make an assumption
 on $g$ analogous to (A1):
 \begin{itemize}
\item[(A3)] If $K=1$, suppose that $g (i) >0$ for 
 all $i \in \N$. If $K \geq 2$, suppose that $g (1) >0$.
 \end{itemize}
 
 We will show that a crucial parameter for the asymptotics of the process is
\begin{equation}
\label{sstar} s^* := \lim_{n \to \infty} G(n) =  \lim_{n \to \infty} \lim_{N \to \infty} G_N (n) .\end{equation}
If (A2) holds, then the $N$-limit exists, and $s^* \in [0,1]$ is well-defined. The value of $s^*$ captures
the `asymptotic atomicity' of $G_N$ in a certain sense.

Before stating our first results, we describe some concrete examples.
For specifying our examples, it is 
  more convenient to work with a version of $\kappa_N$ on ranked sequences,
 namely $\gamma_N$ defined for $i_1 < \cdots < i_K$ by
$\gamma_N ( i_1, \ldots, i_K ) = K! \kappa_N (i_1, \ldots, i_K )$. With this notation, note that
\begin{align}
\label{ggamma} g_N (i) = \frac{1}{K} \left( \sum_{i < i_2 < i_3 < \cdots < i_K} \gamma_N ( i,i_2,i_3,\ldots, i_K) + \sum_{i_2 < i < i_3 < \cdots < i_K}
\gamma_N (i_2, i, i_3, \ldots, i_K) \right.
\nonumber\\ \left. + \cdots +  \sum_{i_2 < i_3 < \cdots < i_K < i}
\gamma_N (i_2, i_3, \ldots, i_K, i) \right) ,\end{align}
where each sum is over $i_2, i_3, \ldots, i_K \in \{1,\ldots,N\}$ satisfying the
given rank constraints.

We describe three examples, by giving the non-zero values of either $\gamma_N$ or $\kappa_N$, as convenient; 
(E1) was discussed in Section \ref{warmup}, while we study examples (E2) and (E3) in detail in Section \ref{main}.

\paragraph{Example (E1).}
 Take $K=1$ and $\gamma_N(k) =1$, i.e., replace the $k$th ranked element only each time. In this case $g_N(k) =1$ as well.

\paragraph{Example (E2).}
 Let $K=2$ and $\gamma_N(1,j) = \frac{1}{N-1}$ for $j \in \{2,\ldots, N\}$, i.e.,
each time we replace the minimal element and one other uniformly chosen point.
This model has been studied by \cite{deboer} and others. From (\ref{ggamma}) we have that in this case $g_N(1) = 1/2$ and, for $i \in \{2,\ldots,N\}$,
$g_N(i) = \frac{1}{2(N-1)}$. Moreover, $g(1) = 1/2$ and $g(i) =0$ for $i \neq 1$.

\paragraph{Example (E3).} (A generalization of (E2).)
Let $K \geq 2$ and let $\phi_N$ be a symmetric
 probability mass function on $\I^{K-1}_N$. Set
 $\kappa_N(1,i_2, \ldots, i_K) = K^{-1} \phi_N (i_2, \ldots, i_K)$.
So now we replace the minimal element and $K-1$ other randomly chosen points,
where the distribution on the `other' points is given by $\phi_N$.
Then $g_N(1) = 1/K$ and, for $i \in \{2, \ldots, N\}$, $g_N(i) = \frac{K-1}{K} f_N(i)$
where $f_N(i) = \sum_{i_3, \ldots, i_K} \phi_N ( i, i_3,\ldots, i_K)$. Write $F_N (n) = \sum_{i=1}^n f_N(i)$.
Assume that $F (n) = \lim_{N \to \infty} F_N(n)$ exists for all $n$,
and set $\alpha = \lim_{n \to \infty} F(n) \in [0,1]$.\\

The assumptions (A1) and (A3) are satisfied by (E2) and (E3), but not (E1), while
(A2) is satisfied by (E1), (E2), and (E3).

We will work with the counting functions defined by (\ref{Ndef}). Our first main result, Theorem \ref{thm00} below,
demonstrates a phase transition in the asymptotic behaviour of the system at the threshold value $s = s^*$;
note that part of the  theorem is the non-trivial statement that $\lim_{N \to \infty} \lim_{t \to \infty} \Exp_N [ C^N_t (s) ]$
exists in $[0,\infty]$.
 
\begin{theo}
\label{thm00}
Suppose that (A1), (A2), and (A3) hold, so that $s^*$ given by (\ref{sstar})
exists in $[0,1]$.
Then
\begin{equation}
\lim_{N \to \infty} \lim_{t \to \infty} \Exp_N [ C^N_t (s) ]  \begin{cases}
< \infty & \textrm{ if } s < s^*  \\
= \infty & \textrm{ if } s >   s^*  \end{cases}.
\end{equation}
\end{theo}

Theorem \ref{thm00} shows that $s^*$ is a threshold value for the model in the sense that
\[ s^* = \sup \{ s \geq 0 : \lim_{N \to \infty} \lim_{t \to \infty} \Exp_N [ C^N_t (s) ] < \infty \} \]
is well defined (with the
convention $\sup \emptyset = 0$). Example (E1) has $s^*=1$, although Theorem \ref{thm00} does not
apply directly (since (A1) and (A3) fail). Example (E2) has $s^*=1/2$, while Example (E3) has $s^* = \frac{1}{K} ( 1 + (K-1) \alpha )$.

For the next result we assume that the distribution $G_N$
given by (\ref{GNdef}) is `eventually uniform'
in a certain sense; roughly speaking we will suppose that $g_N(n) \approx \frac{1-s^*}{N}$
for $n$ large enough. The precise condition that we will use is as follows.
\begin{itemize}
\item[(A4)] Suppose that there exists $n_0 \in \{2,3,\ldots\}$ such that
\[ \lim_{N \to \infty} \sup_{n_0 \leq n \leq N } \left| \frac{ N ( G_N(n) -s^*)}{n-n_0+1} - (1-s^*) \right| = 0 .\]
\end{itemize}
For instance, Example (E2) satisfies condition (A4)
 with $s^* = 1/2$ and $n_0=2$, since $G_N(n) = \frac{1}{2} + \frac{n-1}{2(N-1)}$. In Example (E3), 
 $G_N(n) = \frac{1}{K} + \frac{K-1}{K} F_N(n)$, so that 
 condition (A4) holds if $F_N(n)$ satisfies a similar condition, namely
 \begin{equation}
 \label{evun2}
  \lim_{N \to \infty} \sup_{n_0 \leq n \leq N } \left| \frac{ N ( F_N(n) -\alpha )}{n-n_0+1} - (1-\alpha) \right| = 0 .\end{equation}

Our next result shows the threshold phenomenon at the `$O(N)$' scale.
We can define a threshold parameter
\begin{equation}
\label{sstar2} s^\# := \sup \{ s \geq 0 : \lim_{N \to \infty} \lim_{t \to \infty} \left( N^{-1} \Exp_N [ C^N_t(s) ] \right) = 0 \}.\end{equation}
If both $s^*$ and $s^\#$, as given by (\ref{sstar}) and (\ref{sstar2}) respectively, are well-defined, then
clearly $s^* \leq s^\#$. Theorem \ref{thm3} shows that, under assumption (A4), $s^* = s^\#$; in other words,
the transition is sharp. 
 We use the notation
\begin{equation}
\label{Vdef} V (s) := \begin{cases} 0 & \textrm{ if } s < s^*  \\
\frac{s-s^*}{1-s^*} & \textrm{ if } s \geq  s^*  \end{cases} . \end{equation}

\begin{theo}
\label{thm3}
Suppose that (A1), (A2), (A3), and (A4) hold.
With $V(s)$ as given by (\ref{Vdef}),
we have that for any $s \in [0,1]$,
\begin{align}
\label{mu}
\lim_{N \to \infty} \lim_{t \to \infty} \left( \frac{ \Exp_N [ C^N_t (s) ]}{N} \right)   = V(s) .
\end{align}
\end{theo}

 Another way to interpret Theorem \ref{thm3} is as follows. Let $X_t^*$ denote $X_t^{(M)}$ where $M$ is a random variable with
  $\Pr_N [ M = j] = 1/N$
for $j \in \{1,\ldots,N\}$.
Then
    \[ \Pr_N [ X_t^* \leq s ] = N^{-1} \sum_{i=1}^N \Pr [ X_t^{(i)} \leq s]
    = N^{-1} \Exp \sum_{i=1}^N \1 \{ X_t^{(i)} \leq s \} = N^{-1} \Exp_N [ C^N_t (s) ] ,\]
    by (\ref{Ndef}),
     so that the conclusion of Theorem \ref{thm3} is equivalent to, for $s \in [0,1]$,
  \[ \lim_{N \to \infty} \lim_{t \to \infty} \Pr_N [ X_t^* \leq s ] = V(s) ;\]
  in other words, the marginal distribution of a `typical' point converges
  (as $t \to \infty$ then $N \to \infty$) to a $U[s^*,1]$ distribution.
  Note that some condition along the lines of (A4) is needed for this result to hold:
  see the example in Remark \ref{symm} below.

\begin{rmk}
\label{rho}
In this paper we restrict attention to the case where
the distribution of replacements is $U[0,1]$, but instead of $U_1, U_2, \ldots$
one could take independent copies of a nonnegative
random variable $W$ with distribution function $\rho$.
The results with the $U[0,1]$ distribution immediately generalize to
distributions $\rho$ that are continuous, supported on a single interval, and strictly increasing on
that interval: to see this, note that $(x_1, \ldots, x_N) \mapsto (\rho(x_1), \ldots, \rho(x_N))$
preserves ranks and $\rho(W)$ has a $U[0,1]$ distribution,
so that the dynamics of the process are preserved, up to the change of scale $s \mapsto \rho(s)$.  Thus our results 
immediately extend
to this class of distributions $W$.
\end{rmk}

\begin{rmk}
\label{symm}
Our results can be translated into complementary results by reversing the
ranking and looking at
$N- C_t^N(s)$. Indeed, $N- C_t^N(s)$ counts the number of points in $(s,1]$; 
translating Theorem \ref{thm00} shows that
 under appropriate versions of (A1)--(A3) (in which conditions
 on $g_N(n)$ are replaced by conditions on $g_N(N-n+1)$) 
the threshold
\[ s_* = \inf \{ s \leq 1 : \lim_{N \to \infty} \lim_{t \to \infty} \Exp_N [ N - C_t^N(s) ] < \infty \} \]
is given by
\[ s_* = \lim_{n \to \infty} \lim_{N \to \infty}   G_N (N-n)   .\]
For example, suppose that $K=2$ and we always replace the smallest and the largest points, i.e.,
$g_N(1)=g_N(N) = 1/2$. Then $G_N(n) = (1 + \1\{n = N\})/2$, so that $G(n) = 1/2$ and $s^* = 1/2$.
But also, $s_* = 1/2$. So the expected number of points in any interval not containing $1/2$
remains finite as $N \to \infty$; in other words, the marginal distribution of a typical point
converges to a unit point mass at $1/2$. This example also serves to demonstrate
that Theorem \ref{thm3} may fail if (A4) does not hold.
\end{rmk}

\section{Detailed example: Replace the minimum and one other}
\label{main}
 
In this section we present some more specific results to complement
our general results from Section \ref{model}.
To do so, we 
specialize to the $K=2$ case of Example (E3) from Section \ref{model},
in which we replace the smallest value and choose
the other value to replace from $\{2,\ldots,N\}$ according to
a probability distribution $f_N$. Write
\begin{equation}
\label{FNdef}
 F_N (k) := \sum_{j=2}^k f_N(j),
 \end{equation} 
 for the corresponding
distribution function, adopting the usual convention
that an empty sum is zero, so that $F_N(0) =F_N(1)= 0$.

 In the general set-up of Section \ref{model},
we have $\kappa_N (1, i) = f_N(i)/2$, $g_N (1) =1/2$ and $g_N(i) = f_N(i)/2$ for $i \in \{2,\ldots, N\}$.
Here, assumption (A1) and (A3) are automatically satisfied;
Assumption (A2)  becomes a condition on $f_N$ (or $F_N$), namely
 that 
\begin{equation}
\label{Fdef}
\lim_{N \to \infty} F_N (n) = F(n)
\end{equation}
exists for all $n \geq 2$. The present version of (A2) is then:
\begin{itemize}
\item[(A2$'$)] Suppose that (\ref{Fdef}) holds.
\end{itemize}
Under (A2$'$), 
\begin{equation}
\label{alpha}\alpha := \lim_{n \to \infty}  F (n) \end{equation}
exists in $[0,1]$. Indeed, since $G_N(n) = (1+F_N(n))/2$,
we have that (A2$'$) implies that $s^*$ given by (\ref{sstar})
satisfies $s^* = \frac {1+\alpha}{2}$. 
 
Before stating our results, we comment briefly on their relation to previous work in the literature.
The model of this section includes that studied by de Boer {\em et al.} \cite{deboer} amongst others (see e.g.\ \cite[\S 5.2.5]{jensen});
the model of \cite{deboer} is the special case where $F_N(n) = \frac{n-1}{N-1}$, which
satisfies (A2) with $\alpha =0$. Thus the $\alpha=0$ cases of our results 
are not surprising in view of the (not completely rigorous) arguments in \cite{deboer},
or the heuristic analysis in \cite[\S 5.2.5]{jensen} that
neglects correlations between the $X_t^{(k)}$, but our results
are more general even in the case $\alpha=0$, and we show explicitly the dependence
of the phase transition on $F_N$ via the parameter $\alpha$.
Moreover, one aim of the present work is to
give a more rigorous approach
 to the results of \cite{deboer} in the present considerably more general setting. 

In this setting, the following result is immediate from Theorem \ref{thm00}.

\begin{theo}
\label{thm0}
Suppose that (A2$'$) holds.
Then
\begin{equation}
\lim_{N \to \infty} \lim_{t \to \infty} \Exp_N [ C^N_t (s) ]  \begin{cases}
< \infty & \textrm{ if } s <  \frac{1+\alpha}{2}  \\
= \infty & \textrm{ if } s >   \frac{1+\alpha}{2}  \end{cases}.
\end{equation}
\end{theo}

Similarly, we have the following translation of Theorem \ref{thm3}
into this setting. The appropriate version of condition (A4)
is:
\begin{itemize}
\item[(A4$'$)] Suppose that $F_N$ satisfies (\ref{evun2}).
\end{itemize}

\begin{theo}
\label{thm30}
Suppose that (A2$'$) and (A4$'$) hold.
With $V(s)$ as given by (\ref{Vdef}),
\[
\lim_{N \to \infty} \lim_{t \to \infty} \left( \frac{ \Exp_N [ C^N_t (s) ]}{N} \right)   = V(s) , ~~~ s \in [0,1].
\]
\end{theo}
  
\begin{rmk}
\label{rho2}
If instead of a $U[0,1]$ distribution we use a distribution $\rho$ for
replacement points, as described in Remark \ref{rho},
then the threshold exhibited in Theorem \ref{thm0} becomes
$s^* =\rho^{-1}(\frac{1+\alpha}{2})$;
the inverse $\rho^{-1}$ is well-defined when $\rho$ satisfies
the conditions described in Remark \ref{rho}.
\end{rmk}

Now we move on to our detailed results 
concerning the case $\alpha =0$;
note that $\alpha =0$ if and only if $f_N(n) \to 0$ as $N \to \infty$ for any $n$.
The case $\alpha =0$ includes the discrete uniform case
 (as considered in \cite{deboer})
 in which $f_N(n) = \frac{1}{N-1}$, 
 but includes many other possibilities.
Theorem \ref{thm0} shows that when $\alpha = 0$ the phase transition occurs
at $s^* = 1/2$. 
The next result gives more information, giving an explicit expression
for the limiting equilibrium expectation in the case in which it
is finite.

\begin{theo}
\label{thm0b}
Suppose that  (A2$'$) holds and that $\alpha = 0$.
Then
\begin{equation}
\lim_{N \to \infty} \lim_{t \to \infty} \Exp_N [ C^N_t (s) ]  = \begin{cases}
2 s + \frac{s^2}{1-2s}  & \textrm{ if } 0 \leq s <  1/2   \\
 \infty & \textrm{ if } s \geq   1/2   \end{cases}.
\end{equation}
\end{theo}

We also have the following explicit description of the limit distribution.

\begin{theo}
\label{thm1}
Suppose that  (A2$'$) holds and that $\alpha = 0$.
If $s <   1/2$, then for any $n \in \Z^+$,
\[ \lim_{N \to \infty} \lim_{t \to \infty} \Pr_N [ C^N_t(s) = n ] =\pi^{s} (n),\]
where
\begin{align}
  \label{pi}
  \pi^s (0) & = 1 - 2s ; \nonumber\\
  \pi^s (1) & = 2s - \left( \frac{s}{1-s} \right)^2 ; \nonumber\\
  \pi^s (n) & = \left( 1 - \left( \frac{s}{1-s} \right)^2 \right) \left( \frac{s}{1-s} \right)^{2 (n-1)}, ~~~(n \geq 2) .\end{align}
On the other hand, if $s \geq  1/2$, then for any $n \in \Z^+$, $\lim_{N \to \infty} \lim_{t \to \infty} \Pr_N [ C^N_t(s) \leq n ] = 0$.
\end{theo}

\begin{rmk}
(i) The corresponding stationary probabilities put forward in \cite{deboer}
do not sum to 1 (see equations (10)--(12) in \cite{deboer}).
Our argument is similar (based on the use of the `$N =\infty$' Markov chain)
but we try to give a fuller justification.
(ii)
Let 
\begin{equation}
\label{taudef}
\tau_N(s) := \min \{ t \in \N : C^N_t(s) = 0\}.
\end{equation} 
By standard Markov chain theory,
$\pi_N^s (n) = ( \Exp_N [ \tau_N (s) \mid C^N_0(s) =0] )^{-1}$. So an immediate consequence of Theorem \ref{thm1}
is that (cf equation (16) of \cite{deboer})
\[ \lim_{N \to \infty} \Exp_N [ \tau_N (s) \mid C^N_0(s) =0] = \begin{cases} \frac{1}{1-2s} & \textrm{ if } s <   1/2  \\
\infty & \textrm{ if }   s \geq  1/2  \end{cases}. \]
\end{rmk}

 We  also prove explicit limiting (marginal) distributions  for the lower order statistics
 themselves. 
 We use the notation
 \begin{equation}
 \label{hdef}
 h_n (s) :=
 \begin{cases} 2s & \textrm{ if } n=1 \\
 \left( \frac{s}{1-s} \right)^{2(n-1)} & \textrm{ if } n \geq 2 \end{cases}. \end{equation}

 \begin{theo}
 \label{thm2}
Suppose that   (A2$'$) holds and that $\alpha = 0$.  
Then for   $n \in \N$,  
 \begin{equation}
 \label{theta}
  \lim_{N \to \infty} \lim_{t \to \infty} \Pr_N [ X_t^{(n)} \leq s ] = 
  \begin{cases} 0 & \textrm{ if } s \leq 0 \\
   h_n (s) & \textrm{ if } 0 \leq s \leq 1/2 \\
    1 & \textrm{ if } s \geq 1/2 \end{cases}. \end{equation}
    \end{theo}

The $n=1$ case of (\ref{theta}) says that
the large $N$, long-time distribution of the smallest component approaches a $U[0,1/2]$ distribution. The distributions
arising for $n \geq 2$ are not so standard, but, 
 as $n \to \infty$, they   approach
  a unit point mass
at $1/2$.
 
Also note that Theorem \ref{thm2} yields
convergence of moments of the $X_t^{(n)}$.
  For example,   for any $k \in \N$,
 $\lim_{N \to \infty} \lim_{t \to \infty} \Exp_N [ (X_t^{(1)} )^k ] =2^{-k}/(k+1)$, and for
 $n \geq 2$ and any $k \in\N$,
 \begin{equation}
 \label{hyper}
  \lim_{N \to \infty} \lim_{t \to \infty} \Exp_N [ (X_t^{(n)} )^k ]
 =
2^{-k} - \frac{k 2^{-(2n+k-2)}}{2n+k-2}
 \; {_2 F_1} (2n-2,2n+k-2;2n+k-1; 1/2)  . \end{equation}
 To see this, note that since $X_t^{(n)}$ is uniformly bounded, its moments converge to those of the distribution
 $h_n$ by bounded
convergence, so we have
\begin{align*} \lim_{N \to \infty} \lim_{t \to \infty} \Exp_N [ (X_t^{(n)})^k ] & = k \int_0^{1/2} s^{k-1} \left(
\lim_{N \to \infty} \lim_{t \to \infty} \Pr_N [ X_t^{(n)} > s ] \right) \ud s  \\
& = k \int_0^{1/2} s^{k-1} (1-h_n(s)) \ud s.\end{align*}
When $n=1$ this is $k \int_0^{1/2} (1-2s)s^{k-1} \ud s$ which yields the claimed result.
When $n \geq 2$, using the substitution $u = 2s$, the limit becomes
\[ 2^{-k} - k 2^{-2(n-1)-k} \int_0^1 u^{2(n-1) + k-1} (1 - (u/2))^{-2(n-1)} \ud u ,\]
which gives (\ref{hyper}) via the integral representation of the hypergeometric function.

 \section{Further remarks and open problems}
 \label{open}

 \subsubsection*{A multidimensional model}

 Allowing more general distributions $W$, as described in Remark \ref{rho},
  enables some multi-dimensional
 models to fit within the scope of our results.
 We describe one example.
 Let $Z$ be a 
 uniform random vector on $[0,1]^2$, and let $\| \, \cdot \, \|$ denote the Euclidean norm.
 Starting with $N$ points in $[0,1]^2$, iterate the following Markovian model:
 at each step in discrete time, replace the minimal-ranked
 point, where the ranking is in order of increasing Euclidean distance from the origin,
and another point (chosen uniformly at random) with independent copies of $Z$.
This model corresponds to the model described in Section \ref{main}
but with the $U_i$ replaced by copies
of $W = \| Z \|$, and with $\alpha = 0$.
Elementary calculations show that $\rho (x) := \Pr [ W \leq x ] = \frac{\pi x^2}{4}$
for $x \in [0,1]$ ($\rho(x)$ is more complicated for $x \geq 1$),
so that the phase transition (see Remark \ref{rho2})
occurs at $s^* = \rho^{-1} (1/2) = \sqrt{2/\pi}  \approx 0.80$. See Figure \ref{fig1}
 for a simulation.

 \begin{figure}[!h]
  \center
 \includegraphics[angle=0,width=0.8\textwidth]{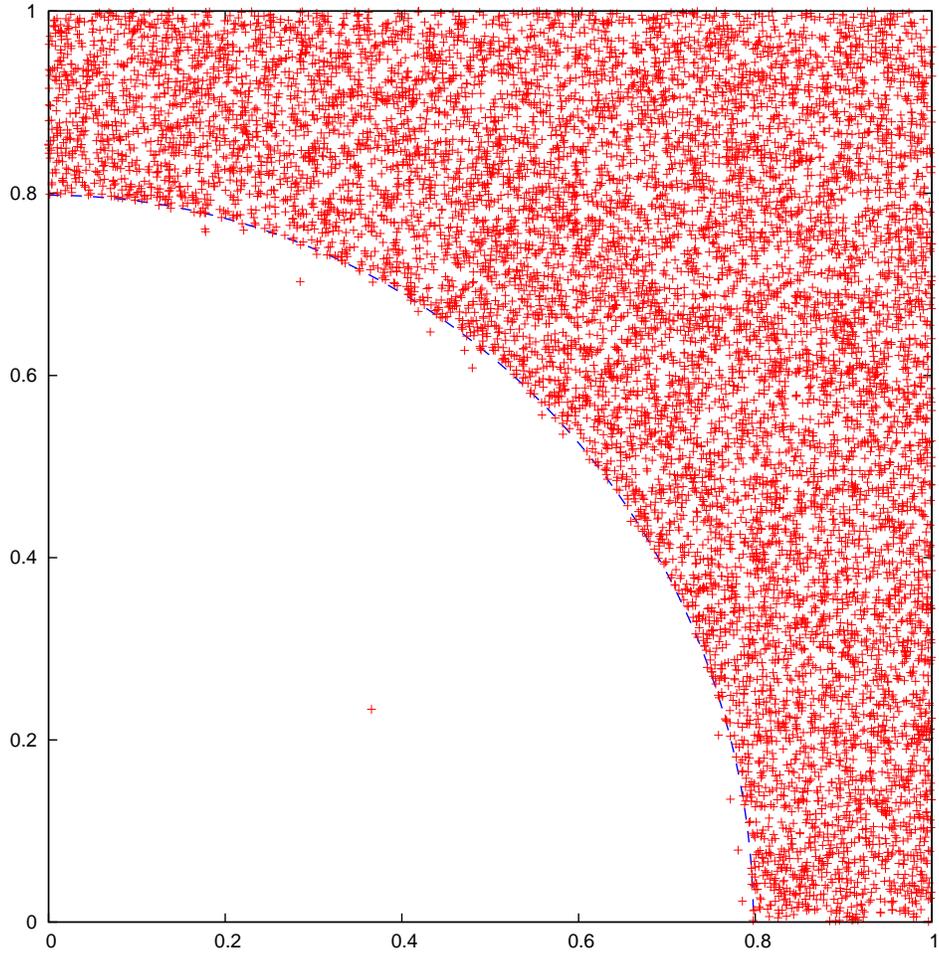}
 \caption{Simulation of the model in which, at each step, the closest point to the origin and one uniformly
 random other point are replaced by independent uniform random points on $[0,1]^2$,
  with $N=10^4$ points and $t=10^6$ steps. The initial distribution was
  $N$ independent uniform points on $[0,1]^2$. Also shown in the figure is part of the
 circle centred at the origin with radius $\sqrt{2/\pi}$.}
 \label{fig1}
 \end{figure}
 
 \subsubsection*{A partial-order-driven process}
 
 Here is a variation on the multidimensional model
 of the previous example governed by a {\em partial order}
 rather than a total order. Again consider a system of $N$ points in $[0,1]^2$.
Consider the co-ordinatewise partial order `$\po$' under which $(x_1,y_1) \po
 (x_2, y_2)$ if and only if $x_1 \leq x_2$ and $y_1 \leq y_2$;
 a point $x$ of a finite set $\mathcal{X} \subset [0,1]^2$ is
 {\em minimal} if and only if there is no $y \in \mathcal{X} \setminus \{ x \}$
 for which $y \po x$. Now define a discrete-time Markov process as follows:
 at each step, replace a minimal element of the $N$ points
 (chosen uniformly at random from amongst all possibilities)
 and a non-minimal element (again, chosen uniformly at random); all new points
 are independent and uniform on $[0,1]^2$. This model seems more difficult to
 study than the previous one, although simulations suggest qualitatively similar
 asymptotic behaviour: see Figure \ref{fig2}.

 \begin{figure}[!h]
  \center
 \includegraphics[angle=0,width=0.8\textwidth]{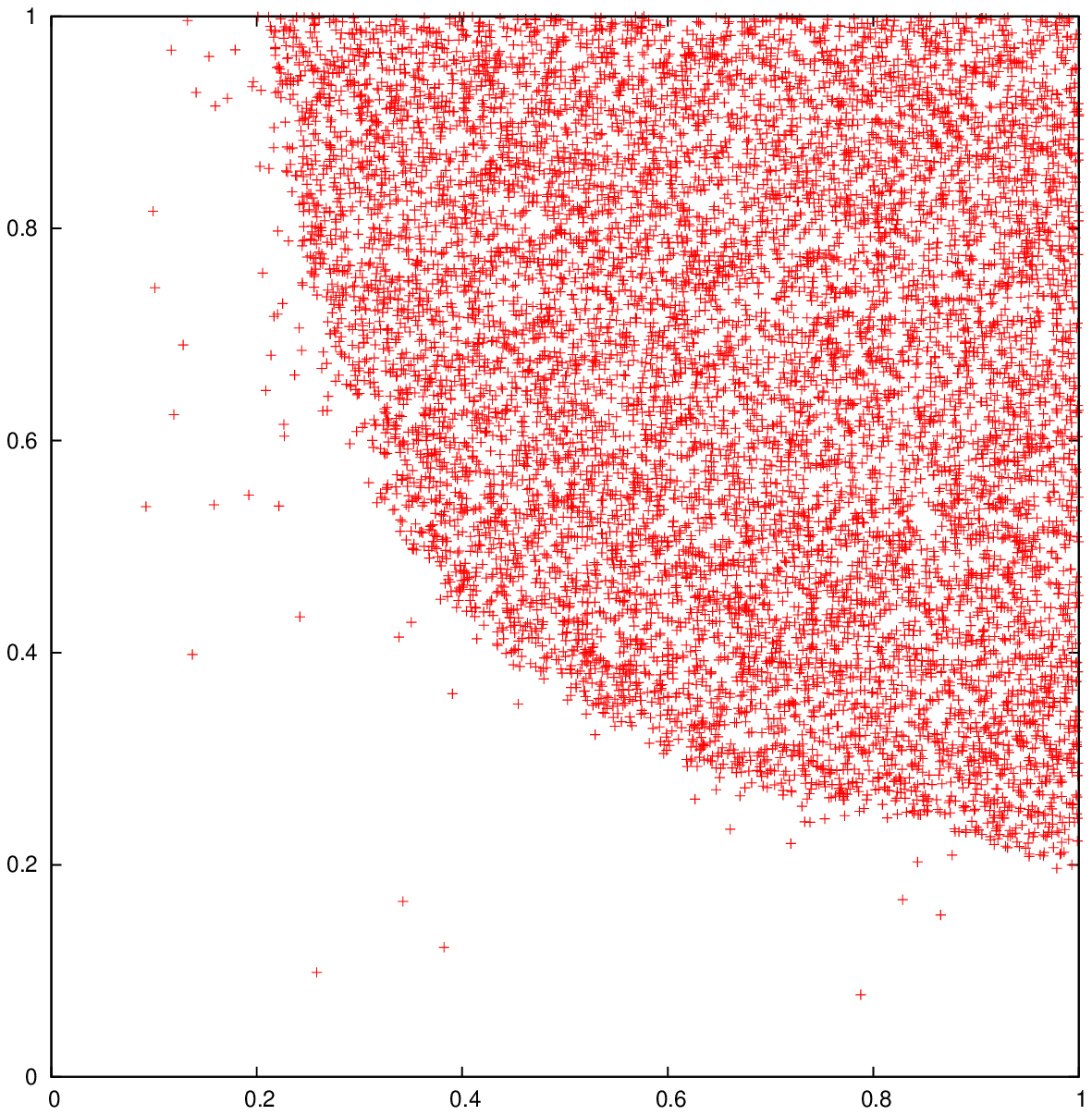}
 \caption{Simulation of the model in which, at each step, one $\po$-minimal element and one non-minimal element
 (each uniformly chosen) are replaced by independent uniform random points on $[0,1]^2$,
  with $N=10^4$ points and $t=10^6$ steps. The initial distribution was
  $N$ independent uniform points on $[0,1]^2$. Can the threshold curve be characterized?}
 \label{fig2}
 \end{figure}

 \subsubsection*{A repeated beauty contest}
 
 We describe a process of a different
 flavour to those previously considered, in which the update
 rule depends not only on the ranks of the points; this is a variation on
 a {\em Keynesian beauty contest}.
 
  Fix a parameter $p >0$. Start with a uniform array of 
$N$ elements on $[0,1]$. At each step, 
compute the mean $\mu$ of the $N$ elements,  
 and replace by a $U[0,1]$
random variable the element that is farthest (amongst  all the $N$ points)
from $p \mu$.
Thus at each step, either the minimum or maximum is replaced, depending
on the current configuration.  

 This is related to the ``$p$-beauty contest'' \cite[p.\ 72]{moulin} in which $N$ players  
 choose a number between $0$ and $100$, the winner being
the player whose choice is closest to $p$ times the average of all the $N$ choices. 
The stochastic process described above is a repeated, randomized version of this 
game (without any learning, and with random player behaviour) in which the worst performer is replaced by a new player.

According to simulations and heuristic considerations,
the equilibrium distribution of a typical point
approaches, for large $N$,
a point mass at $0$ ($1$) in the case  $p<1$ ($p>1$).
The case $p=1$ is more subtle, and is reminiscent of a  P\'olya urn.
Stochastic approximation ideas (see e.g.\ \cite{pemantle})
may be relevant in studying this model.

\section{General thresholds: Proofs for Section \ref{model}}
\label{sec:genproofs}

\subsection{Overview}

This section contains the proofs of our general results from Section \ref{model},
and is arranged as follows. In Section \ref{sec:genfiniteN} we give a basic result
on the Markov chains $C_t^N(s)$. To study the $N \to \infty$ asymptotics of these
Markov chains, at least when $s < s^*$,
 we introduce an `$N = \infty$' Markov chain $C_t(s)$. In Section
\ref{sec:genlimitchain} we show that we can define $C_t(s)$
in a consistent way, and we prove some of its basic properties. In Section \ref{sec:asymptotics}
we relate the asymptotic properties of the finite-$N$ chains $C_t^N(s)$ with $s < s^*$
 to the chain $C_t(s)$,
making use of our technical results from Section \ref{sec:limits}.
Then in Section \ref{sec:proofs} we complete the proofs
of Theorems \ref{thm00} and \ref{thm3}.

\subsection{The Markov chain $C^N_t(s)$}
\label{sec:genfiniteN}
 
We have the following basic result. 

\begin{lm}
\label{lem0}
Suppose that (A1)  holds. Fix $N \in \N$. Suppose that   $s \in (0,1)$.  
Then $C^N_t(s)$
is an irreducible, aperiodic Markov chain on $\{0,1,\ldots, N\}$ with uniformly bounded jumps:
 $\Pr_N [ | C^N_{t+1} (s) - C^N_t (s) | > K ] = 0$.
There
exists a unique stationary distribution $\pi_N^s$, with $\pi_N^s (n) >0$ for all
$n \in \{0,1,\ldots,N\}$ and $\sum_{n =0}^N \pi_N^s(n) =1$, such that
\begin{equation}
\label{stat} \lim_{t \to \infty} \Pr_N [ C^N_t(s) = n ] =\pi^s_N(n) ,\end{equation}
for any initial distribution $C^N_0(s)$.
Moreover,
 \begin{equation}
 \label{meanconv}
\lim_{t \to \infty} \Exp_N [C^N_t(s)] = \sum_{n=1}^N n \pi_N^s (n).
\end{equation}
Finally, with $G_N$ as defined at (\ref{GNdef}),
the one-step mean drift of $C^N_t(s)$ is given by
\begin{equation}
\label{gendrift}
\Exp_N [ C^N_{t+1} (s) - C^N_t (s) \mid C^N_t (s) = n] = K ( s -  G_N (n) ).
\end{equation}
\end{lm}

Note that the degenerate cases $s \in \{0,1\}$ are excluded
from Lemma \ref{lem0}:
$C^N_t(1) =N$ a.s.\ for all $t$,
while 
 $C^N_t (0) \to 0$ a.s.\ as $t \to \infty$ for any initial distribution $C_0^N(0)$.

\begin{proof}[Proof of Lemma \ref{lem0}.]
Note that
\[ \{C^N_t(s) = n\} = \{ X_t^{(1)} \leq s, \ldots, X_t^{(n)} \leq s, X_t^{(n+1)} > s, \ldots, X_t^{(N)} > s \} ;\]
the distribution of $C^N_{t+1} (s)$ depends only on   $(X_t^{(1)}, \ldots, X_t^{(N)})$
through events of the form on the right-hand side of the last display.
Specifically, given $C_t^N(s) =n$, we have that the increment  $C^N_{t+1}(s)-C^N_t(s)$
is the number of the $K$ new $U[0,1]$-distributed
points that fall in $[0,s]$ minus the
number of the $K$ points selected for removal whose rank was at most $n$. That is,
with the notation 
\begin{align*} B_{t+1} (s) & := \# \{ i \in \{Kt+1 , \ldots, Kt+K \} : U_i \in [0,s] \}, \\
A^N_{t+1} (n) & := \# \{ i \in \{ 1, \ldots , K \} : R^N_i(t+1) \leq n \} ,
\end{align*}
we have that, given $C_t^N(s) =n$,
\begin{equation}
\label{geninc}
 C^N_{t+1}(s) = n + B_{t+1} (s) - A^N_{t+1} (n) 
 .\end{equation} 
Thus, given $C_t^N(s)=n$,
 the increment depends only on $n$ and the variables $R^N(t+1)$, $U_{Kt+1}, \ldots, U_{Kt+K}$,
which are all independent of $C^N_t(s)$.
 This demonstrates the Markov property. 

The bounded jumps property
is clear by definition, and can also been seen from (\ref{geninc}).
 To show irreducibility and aperiodicity, 
  we show that
\begin{align*} \Pr_N [ C^N_{t+1} (s) = n \mid C^N_t(s) = n ] & > 0, ~~~ ( n \in \{0,1,\ldots, N\}),\\
\Pr_N [ C^N_{t+1} (s) = n+1 \mid C^N_t(s) = n ] & > 0, ~~~ ( n \in \{0,1,\ldots, N-1\}),\\
\Pr_N [ C^N_{t+1} (s) = n-1 \mid C^N_t(s) = n ] & > 0, ~~~ ( n \in \{1,2,\ldots, N\}).
\end{align*}
Since $B_{t+1}(s)$ and $A_{t+1}^N(n)$ are independent given $C^N_t(s)=n$, it suffices
to show that $\Pr_N [ B_{t+1} (s) = i ] >0$ for any $i \in \{0,1,\ldots, K\}$, and 
that $\Pr_N [ A^N_{t+1} (n) = i ] >0$ for: (i) $i \in \{0,1\}$ if $K=1$; or (ii) $i =1$ if $K \geq 2$.
Then the intersection of two independent events of positive probability will yield
any increment of $C_t^N(s)$ in $\{-1,0,1\}$, as required.
First consider $B_{t+1}(s)$: this has a ${\rm Bin }(K,s)$ distribution,
and so takes any value in $\{0,1,\ldots, K\}$ with positive probability, provided $s \in (0,1)$.
Now consider $A_{t+1}^N (n)$.
Then
\begin{equation}
\label{ANt}
 A^N_{t+1} (n) = \sum_{i=1}^n \1 \{ i \in \{ R^N_1(t+1), \ldots, R^N_K(t+1) \} \} .\end{equation}
It follows from (\ref{ANt}) that, for $n \geq 1$,
 $ A^N_{t+1} (n) \geq \1 \{ 1 \in \{ R^N_1(t+1), \ldots, R^N_K(t+1) \} \}$, so  
\[ \Pr_N [ A^N_{t+1} (n) = 1 ] \geq \Pr_N [  1 \in \{ R^N_1(t+1), \ldots, R^N_K(t+1) \} ] \geq \Pr_N [ R^N_1 = 1] .\]
This latter probability is $g_N (1)$, which is positive by (A1). 
This completes the proof of irreducibility and aperiodicity in the case $K \geq 2$; it remains to
show that $\Pr_N [ A^N_{t+1} (n) = 0 ] > 0$ for $n \geq 0$ when $K=1$. Using the $K=1$ case of (\ref{ANt}), we obtain
\begin{align*}  \Pr_N [ A^N_{t+1} (n) = 0 ]  = \Pr_N [   R^N_1(t+1) > n  ] = 1 - G_N (n) ,\end{align*}
by (\ref{GNdef}),
and $1- G_N(n) >0$
 since (A1) implies that in this case $g_N(i) >0$ for some $i > n$.

Thus the Markov chain is irreducible and aperiodic; it
 has a finite state-space, and so standard Markov chain theory implies the existence
 of a unique stationary distribution, for which (\ref{stat}) holds.
Moreover,  since $C^N_t(s)$ is bounded by $N$, (\ref{meanconv}) follows from (\ref{stat}).

Finally we prove the statement (\ref{gendrift}).
We take expectations in (\ref{geninc}); $B_{t+1}(s)$
has mean $Ks$, and taking expectations in (\ref{ANt}) we obtain
\begin{align*} 
\Exp [ A_{t+1}^N (n) ]
 =   \sum_{i=1}^n \Pr [  i \in \{ R^N_1(t+1), \ldots, R^N_K(t+1) \}  ]  
  = K \sum_{i=1}^n \Pr [ R^N_1 = i],\end{align*}
by exchangeability. Thus from (\ref{GNdef}) we obtain (\ref{gendrift}).
  \end{proof}
 
 A key step in our analysis is to study the stationary distributions $\pi_N^s$ of the Markov chains $C^N_t(s)$, $s \in (0,1)$,
 whose existence is proved in Lemma \ref{lem0}.
 We consider $\pi_N^s$ as $N \to \infty$.
 One tool that we will use is a Markov chain $C_t(s)$ on the whole of $\Z^+$
 that can be viewed in some sense as the $N \to \infty$ limit of the
 Markov chains $C^N_t (s)$: this Markov chain we call
 the `$N = \infty$' chain, and we describe it in Section \ref{sec:genlimitchain}; 
 in Section \ref{sec:asymptotics} we make   precise
 the sense in which the `$N=\infty$' chain is a limit
 of the finite-$N$ chains.

 \subsection{The `$N = \infty$' chain $C_t(s)$}
 \label{sec:genlimitchain}

Our asymptotic analysis makes use of an `$N=\infty$'
analogue of the Markov chain $C^N_t(s)$.
The case $N = \infty$ does not make sense directly in terms of the original model $X_t$,
 but (A2) can be used to
 define a Markov chain on the whole of $\Z^+$, which we can relate to our finite-$N$ Markov chains,
 at least when $s < s^*$. 

 We  use $C_t(s)$ to denote our new Markov chain, now defined on the whole of $\Z^+$, and
 we write $\Pr$ for the associated probability measure and $\Exp$ for the corresponding expectation.
 The idea is to define transition probabilities via 
 \[ \Pr [ C_{t+1} (s) = m \mid C_t (s) = n] = \lim_{N \to \infty}
 \Pr_N [ C^N_{t+1} (s) = m \mid C^N_t (s) = n] ;\]
 to show that this is legitimate under suitable assumptions, we need the following result.
 
 \begin{lm}
 \label{limitprobs}
 Suppose that (A2) holds. Let $s \in [0,1]$. Then for any $n, m \in \Z^+$,
 \[ p^s (n,m) := \lim_{N \to \infty} \Pr_N [ C_{t+1}^N (s) = m \mid C_t^N (s) = n] \]
 is well-defined, and $\sum_{m \in \Z^+} p^s (n,m) = 1$.
 \end{lm}
 \begin{proof} 
 We show that the increment 
distribution, conditional on $\{ C^N_{t} = n\}$,
 given by (\ref{geninc}) in the finite $N$ case converges (as $N \to \infty$), using assumption (A2),
to an appropriate limiting distribution, which will serve as the
increment distribution $p^s ( n, \, \cdot \,)$.
This convergence is clear for the term $B_{t+1}(s)$, which has no $N$-dependence.
Moreover, given $C^N_t = n$, the terms $B_{t+1}(s)$ and $A^N_{t+1}(n)$ are independent.
Thus it suffices to show that $A^N_{t+1}(n)$ converges in distribution
to a proper random variable. We show that this follows from (A2), although  care is needed 
to correctly account for  lost
mass in (A2).

To proceed, it is useful to introduce more notation. Let $R = (R_1, \ldots, R_K)$ denote the $N \to \infty$ distributional
limit of $R^N$: given (A2), this limit exists but is not necessarily a proper
distribution on $\N^K$, but we recover a proper distribution by expanding the state-space
to $( \N \cup \{ \infty \} )^K$. Thus components of $R$ may take the value $\infty$:
this cannot be directly interpreted in terms of rank distributions, but is convenient
for correctly accounting for the lost mass in (A2). Concretely, the distribution of $R$ is given,
for any $k \leq K$ and any distinct $i_1, i_2, \ldots , i_k \in \N$,
\begin{align}
\label{Rdist} & \Pr [ R_1 = i_1, \ldots, R_k = i_k,   R_{k+1} = \infty, \ldots, R_K = \infty ] \nonumber\\
& ~~~ = \lim_{N \to \infty} \Pr_N [ R_1^N =i_1, \ldots, R_k^N = i_k ]
- \sum_{i_{k+1} \in \N} \cdots \sum_{i_K \in \N} \lim_{N \to \infty} \Pr_N [ R_1^N=i_1, \ldots, R_K^N = i_K ] \nonumber\\
&~~~ = \kappa (i_1, \ldots, i_k ) - \sum_{i_{k+1} \in \N} \cdots \sum_{i_K \in \N} \kappa (i_1, \ldots , i_K ) ,\end{align}
using (A2). Note that since $R^N$ is exchangeable on $\{1,\ldots, N\}^K$, it follows that
$R$ is exchangeable on $( \N \cup \{ \infty \})^K$.

Now we can define the $N=\infty$ analogue of $A^N_{t+1} (s)$ to
be an independent copy of $\# \{ i \in \{1,\ldots,K\} : R_i \leq n \}$, i.e.,
for $R(t+1) = (R_1(t+1), \ldots, R_K(t+1))$ an independent copy of $R$, with distribution
given by (\ref{Rdist}), we take
\[ A_{t+1} (n) := \sum_{i=1}^K \1 \{ R_i (t+1) \leq n\} .\]
Then we can construct $C_t(s)$ via its increments
\begin{equation}
\label{infinc}
C_{t+1} (s) - C_t (s) = B_{t+1} (s) - A_{t+1} ( C_t(s) ) .
\end{equation}
Since $R^N$ converges in distribution to $R$ as $N \to \infty$,  
$A^N_{t+1} (n) = \sum_{i=1}^K \1 \{R^N_i(t+1) \leq n\}$ converges
in distribution to $A_{t+1}(n)$;
specifically, using exchangeability,
\begin{align*} \Pr_N [ A^N_{t+1} (n) = k ] & = {K \choose k} \Pr_N [ R_1^N \leq n ,\ldots, R_k^N \leq n , R_{k+1}^N > n,
\ldots, R_K^N > n ] \\
& \to {K \choose k} \Pr  [ R_1  \leq n ,\ldots, R_k  \leq n , R_{k+1}  > n,
\ldots, R_K > n ] ,\end{align*}
as $N \to \infty$. This completes the proof.
\end{proof}

The following result gives some basic properties of the Markov chain defined above.

\begin{lm}
\label{genlimitchain}
Suppose that (A2) and (A3) hold.
Then for any $s \in (0,1)$, $C_t(s)$ is 
an irreducible, aperiodic Markov chain on $\Z^+$, with uniformly bounded jumps:
 $\Pr [ | C_{t+1} (s) - C_t (s) | > K ] = 0$.
The one-step mean drift of $C_t(s)$ is given by
\begin{equation}
\label{gendrift2}
\Exp [ C_{t+1} (s) - C_t (s) \mid C_t (s) = n] = K ( s -  G (n) ).
\end{equation}
\end{lm}
\begin{proof}
The boundedness of the increments follows from the construction in (\ref{infinc}).
The irreducibility and aperiodicty  follow
 from a similar argument to that used  in the proof of Lemma \ref{lem0}
 in the finite-$N$ case, now using (A3) in place of (A1). The drift (\ref{gendrift2}) also follows
 similarly to the proof of (\ref{gendrift}) in Lemma \ref{lem0}; in the present case
 \[ \Exp [ A_{t+1} (n) ] = \sum_{i=1}^n \Pr [ i \in \{ R_1, \ldots, R_K \} ] =
 K \sum_{i=1}^n \Pr [ R_1 = i ] ,\]
 by exchangeability of $R$ (see the comment after (\ref{Rdist})). But $\sum_{i=1}^n \Pr [ R_1 =i ] = G(n)$,
 by (A2) and the definition of $G(n)$ at (\ref{Gdef}).
 \end{proof}

\subsection{Large-$N$ asymptotics}
\label{sec:asymptotics}

 We show that properties of the Markov chains $C_t^N(s)$, described in Section \ref{sec:genfiniteN},
 in the large $N$ limit can (at least when $s < s^*$)
  be described using the `$N=\infty$' Markov chain $C_t(s)$, described in Section \ref{sec:genlimitchain}.
  The main tool is Theorem \ref{mclim} stated and proved in Section \ref{sec:limits}.
Recall the definition of $\pi_N^s$ from Lemma \ref{lem0}.

 \begin{lm}
 \label{lem:limits}
 Suppose that (A1), (A2), and (A3) hold, and that $s \in (0,s^*)$. There
exists a unique stationary distribution $\pi^s$ for $C_t(s)$, with $\pi^s (n) >0$ for all
$n \in \Z^+$ and $\sum_{n \in \Z^+} \pi^s(n) =1$, such that
\begin{equation}
\label{stat2} \lim_{t \to \infty} \Pr [ C_t(s) = n ] =\pi^s(n) ,\end{equation}
for any initial distribution $C_0(s)$. In addition, the following results hold.
 \begin{itemize}
 \item[(a)] There exist $c>0$ and $C< \infty$ such that, for all $n \in \Z^+$,
 $\pi^s_N (n) \leq C \re^{-cn}$ and $\pi^s(n) \leq C \re^{-cn}$.
 \item[(b)] For any $n \in \Z^+$, $\lim_{N \to \infty} \pi^s_N(n) = \pi^s(n)$.
 \item[(c)] As $t \to \infty$, $\Exp_N [ C^N_t (s) ] \to \sum_{n=0}^N n \pi_N^s (n)$
 and $\Exp [ C_t (s) ] \to \sum_{n \in \Z^+} n \pi^s (n) < \infty$.
 \end{itemize}
 \end{lm}
\begin{proof}
We will show that we can apply Theorem \ref{mclim} with $Y_t^N = C_t^N(s)$, $Y_t = C_t(s)$,
$S_N = \{0,1,\ldots,N\}$, and $S = \Z^+$.
Since $s \in (0,1)$ and (A1) holds, Lemma \ref{lem0} shows that $C_t^N(s)$ is an irreducible
Markov chain on $\{0,1,\ldots,N\}$, while, since (A2) and (A3) hold, Lemma \ref{genlimitchain}
implies that $C_t(s)$ is an irreducible Markov chain on $\Z^+$.
Lemmas \ref{lem0} and \ref{genlimitchain} also imply that
the increments of $C_t^N(s)$ and $C_t(s)$ are uniformly bounded in absolute value (by $K$) almost surely.
Thus (\ref{jumps}) holds.

Next we verify the drift conditions in (\ref{foster}).
 Since $s <  s^*$,
 there exists $\eps>0$
 such that $s < s^* -2\eps$. First consider the finite-$N$ case. By (A2) and the definition
 of $s^*$ at (\ref{sstar}), given $\eps$, we can take $N_0$ and $n_0$ such that for any
 $N \geq N_0$ and any $n \geq n_0$,
 \[ G_N (n) > s^* - \eps > s + \eps .\]
 So we have from (\ref{gendrift}) that,   for
 all $N \geq N_0$ and $n \geq n_0$,
 \[ \Exp_N [ C^N_{t+1}(s) - C^N_t(s) \mid C^N_t(s) = n ] \leq - \eps K. \]
 A similar argument holds for $C_t(s)$, using (\ref{gendrift2}). Thus (\ref{foster}) is satisfied.
  Finally, we verify (\ref{problim}) by Lemma \ref{limitprobs}.
 Thus Theorem \ref{mclim} applies, yielding the claimed results.
\end{proof}

The next result deals with the case $s > s^*$. Recall that  $\tau_N(s)$
defined by  (\ref{taudef}) denotes the time of the first return of $C^N_t(s)$ to $0$.

\begin{lm}
\label{lem:null}
 Suppose that (A1), (A2), and (A3) hold, and that $s > s^*$. Then $\lim_{N \to \infty} \Exp_N [ \tau_N (s) ] = \infty$.
\end{lm}
\begin{proof}
 Suppose that $s > s^*$. Then, for  some $\eps>0$,
 $s-s^* - \eps > \eps$.
 Fix $x \in \N$. Since, by (A2),
  $\lim_{N \to \infty}
 G_N(n) = G(n) \leq s^*$ for any $n$, we can find $N_0(x)$ such that
 $G_N (n) \leq s^* + \eps$ for any $n \leq x$ and any $N \geq N_0(x)$.
 Hence, by (\ref{gendrift}),
 \begin{equation}
 \label{posdrift} \Exp_N [ C^N_{t+1}(s) - C^N_t (s) \mid C^N_t(s) = n] \geq K \eps ,\end{equation}
 for any $N \geq N_0(x)$ and any $n \leq x$, where $\eps>0$ does not depend on $x$.
 We show that (\ref{posdrift}) implies that $C_t^N(s)$ has a positive probability
 (uniform in $x$) of reaching $x$ before returning to $0$, which will imply the result.
 It suffices to suppose that $C_0^N(s) \geq 1$.
 
 To ease notation,   write $\tau:= \tau_N(s)$ for the remainder of this proof.
 To estimate the required hitting probability, set $W_t := \exp \{ -\delta C^N_t (s) \}$, for $\delta>0$
 to be chosen later. Now
 \begin{align*}
 W_{t+1} - W_t & = \exp \{ - \delta C_t^N(s) \} \left( \exp \{ - \delta ( C_{t+1}^N(s) - C_t^N (s) ) \} -1 \right) \\
 & \leq \exp \{ - \delta C_t^N(s) \} \left( - \delta ( C_{t+1}^N(s) - C_t^N (s) ) + M \delta^2 \right) ,
 \end{align*}
 for some absolute constant $M$, using the fact that the increments of $C^N_t(s)$ are uniformly bounded.
 Taking expectations and using (\ref{posdrift}), we have that, on $\{ C^N_t(s) \leq x \}$,
 \[ \Exp_N [ W_{t+1} - W_t \mid C^N_t(s) ] \leq \exp \{ - \delta C_t^N(s) \} \left( - K \eps \delta + M\delta^2 \right) \leq 0 ,\]
 for $\delta \leq \delta_0$ small enough, where $\delta_0>0$ depends only on $\eps$ and not on $x$ or $N$.
 Let $\nu_x := \min \{ t \in \Z^+ : C_t^N(s) \geq x \}$. Then we have shown that $W_{t \wedge \tau \wedge \nu_x }$
 is a nonnegative supermartingale, which converges a.s.\ to $W_{\tau \wedge \nu_x}$. It follows that
 \begin{align*}
 \re^{-\delta} \geq W_0 \geq \Exp_N [ W_{\tau \wedge \nu_x } ]
 \geq \Pr_N [ \tau < \nu_x ]  
 ,\end{align*}
 so that $\Pr_N [ \nu_x < \tau ] \geq   1 - \re^{-\delta} =: p$, 
 where $p>0$ does not depend on $x$ or on $N$.
 The fact that $C^N_t(s)$ has increments of size at most $K$ implies that
 on $\{ \nu_x < \tau\}$ we have $\{ \tau \geq x/K\}$. Hence $\Pr_N [ \tau \geq x/K ] \geq \Pr_N [ \nu_x < \tau ]
 \geq p$,
 so that $\Exp_N [\tau] \geq px/K$ for all $N \geq N_0 (x)$.
 Since $x$ was arbitrary, the result follows.
\end{proof}

\subsection{Proofs of Theorems \ref{thm00} and \ref{thm3}}
\label{sec:proofs}

 \begin{proof}[Proof of Theorem \ref{thm00}.]
  First suppose that $s < s^*$.
  Then Lemma \ref{lem:limits} applies.
  By Lemma  \ref{lem:limits}(a),
 $\pi^s_N (n) \leq C \re^{-cn}$ where $C < \infty$ and $c>0$ do not depend
 on $N$ or $n$.
 In particular, for any $p >0$, $\sup_N \sum_{n \in \Z^+} n^p \pi_N^s (n) < \infty$.
 Moreover, by  Lemma  \ref{lem:limits}(b), $\pi_N^s (n) \to \pi^s (n)$ as $N \to \infty$.
 Hence for any $p>0$, by uniform integrability, $\sum_n n^p \pi_N^s (n) \to \sum_n n^p \pi^s(n)$
 as $N \to \infty$. Together with  Lemma  \ref{lem:limits}(c), this implies that, for $s <  s^*$,
 \begin{equation}
 \label{pimean}
  \lim_{N \to \infty} \lim_{t \to \infty} \Exp_N [ C^N_t (s) ]
 = \lim_{ N \to \infty} \sum_{n\in \Z^+} n  \pi_N^s (n) = \sum_{n \in \Z^+} n  \pi^s(n) < \infty.\end{equation}

 Next suppose that $s >  s^*$.
Then, for fixed $x>0$ and some $\eps>0$ (not depending on $x$)
 we have that (\ref{posdrift}) holds 
 for any $N \geq N_0(x)$ and any $n \leq x$.
 On the other hand, if $C^N_t(s) >x$, we have that $C^N_{t+1} (s) \geq x-K$
 (by bounded jumps). It follows that
 \[ \Exp_N [ C^N_{t+1} (s) - C^N_t(s) \mid C^N_t (s) ] \geq K(1+\eps) \1 \{ C^N_t(s) \leq x \} -K  .\]
 Taking expectations implies that
 \[ \Exp_N [ C^N_{t+1} (s) ] - \Exp_N [ C^N_t (s) ] \geq K(1+\eps) \Pr_N [ C^N_t(s) \leq x ] -K .\]
 By (\ref{meanconv}),
 the left-hand side of the last display tends to $0$ as $t \to \infty$. 
 It follows that, for some $\delta>0$ that depends on $\eps$ but not on $x$,
 $\Pr_N [ C^N_t(s) \geq x ] \geq \delta$ for all $t$ large enough. Hence
$\Exp_N [ C^N_t(s) ] \geq x \delta$, 
 for all $N$ and $t$ sufficiently large, 
 Since $x$ was arbitrary, and $\delta$ did not depend on $x$,
  the second part of the theorem follows.
  \end{proof}

\begin{proof}[Proof of Theorem \ref{thm3}.]
For $s <  s^*$, the statement follows immediately from Theorem \ref{thm0}.

Suppose that $s > s^*$. For the duration of this proof, we write $\tau$ for $\tau_N(s)$ to ease notation.
In this case, Lemma \ref{lem:null} applies, showing that $\lim_{N \to \infty} \Exp_N [\tau] = \infty$.
We claim that $C_t^N(s)$ is asymptotically null in the sense that, for any $n \in \Z^+$,
\begin{equation}
\label{eq5}
\lim_{N \to \infty} \lim_{t \to \infty} \Pr_N [ C^N_t (s) \leq n ] = 0. \end{equation}
Indeed, (\ref{eq5}) follows
from Lemma \ref{lem:null}
and the occupation-time representation for the stationary distribution
of an irreducible, positive-recurrent Markov chain  (see e.g.\ \cite[Corollary I.3.6, p.\ 14]{asmussen})
 gives: 
  \[ \lim_{t \to \infty} \Pr_N [ C^N_t (s) \leq n ] = \sum_{x = 0}^n \pi^s_N (x) = \frac{ \Exp_N \sum_{t=0}^{\tau -1} \1 \{ C_t^N (s) \leq n \} }
  { \Exp_N [ \tau ] } ;\]
  in the final fraction, the denominator tends to infinity with $N$ (by Lemma \ref{lem:null})
  while the numerator is uniformly bounded in $N$ since the expected number of visits
  to any bounded interval stays bounded, by irreducibility (uniform in $N$). Thus (\ref{eq5}) holds for $s > s^*$.

Taking expectations in (\ref{gendrift}) yields
\begin{equation}
\label{eq6} \Exp_N [ C^N_{t+1} (s) ] - \Exp_N [ C^N_t (s) ] = K s  - K \Exp_N [ G_N(C^N_t(s))  ] .\end{equation}
The left-hand side of (\ref{eq6}) tends to $0$ as $t \to \infty$ by (\ref{meanconv}).
Also, for $n_0$ as in (A4),
\[  \lim_{N \to \infty} \lim_{t \to \infty} \Exp_N [ G_N ( C^N_t (s) ) \1 \{ C_t^N (s) < n_0 \} ] \leq  \lim_{N \to \infty} \lim_{t \to \infty} \Pr_N [ C^N_t (s) \leq n_0 ] ,\]
which is $0$ by (\ref{eq5}). 
Hence, taking limits in (\ref{eq6}), we obtain
\begin{align*}
 s    = \lim_{N \to \infty} \lim_{t \to \infty} \Exp_N [ G_N(C^N_t(s))  \1 \{ C^N_t (s) \geq  n_0 \} ] .
\end{align*}
By condition (A4), a.s.,
\[ G_N(C^N_t(s))  \1 \{ C^N_t (s) \geq  n_0 \} = s^* + (1-s^*) \frac{C^N_t(s)}{N} + \eps_N ,\]
where $\eps_N = o(1)$ as $N \to \infty$, uniformly in $C^N_t(s)$ (and hence uniformly in $t$). Thus
\[ s = \lim_{N \to \infty} \lim_{t \to \infty} \left( s^* + (1-s^*) N^{-1} \Exp_N [  C^N_t(s)  ] \right) ,\]
which yields the result (\ref{mu}) for $s > s^*$.

Finally, suppose that $s = s^*$.
Then $C_t^N(s) \leq C_t^N(r)$ for
all $t$,  $N$, and  $r > s^*$.
Hence
\begin{align*} \lim_{N \to \infty} \lim_{t \to \infty} \left( N^{-1} \Exp [ C_t^N(s) ] \right)
& \leq \lim_{r \downarrow s} \lim_{N \to \infty} \lim_{t \to \infty} \left( N^{-1} \Exp [ C_t^N(r) ] \right) \\
& = \lim_{r \downarrow s} V(r) ,\end{align*}
by the previous part of the proof,
since $r > s^*$. The latter limit is $V (s^*) =0$,
so the result (\ref{mu}) is proved for $s = s^*$ as well.
\end{proof}

 \section{Proofs for Section \ref{main}}
\label{sec:proofs2}

\subsection{Overview}

In this section we work with the $K=2$ case of Example (E3), working towards proofs
of the results in Section \ref{main}. The organization of the section broadly mirrors
that of Section \ref{sec:genproofs}. In Sections \ref{sec:finiteN} and \ref{sec:limitchain}
we return to the finite-$N$ Markov chain $C^N_t(s)$ and the limit chain
$C_t(s)$, respectively, describing their properties more explicitly in this special
case, for which exact computations are available. Then in Section \ref{sec:remains}
we give the proofs of our remaining results, Theorems \ref{thm0b}, \ref{thm1}, and \ref{thm2}.

\subsection{The Markov chain $C^N_t(s)$}
\label{sec:finiteN}

 For this model, the following result complements the general 
 Lemma \ref{lem0}.  Write
$p^s_N(n,m) := \Pr_N [ C^N_{t+1} (s) = m \mid C^N_t(s) = n]$.
Recall the definition of $F_N$ from (\ref{FNdef}), and that $F_N(0) = F_N(1)=0$.
In the case where $F_N(n) = \frac{n-1}{N-1}$,
$p^s_N(n,m)$ was written down in equations (1)--(3) in \cite{deboer}. 

\begin{lm}
\label{lem1}
For any $s \in [0,1]$, $(C^N_t(s))_{t \in \Z^+}$ is a Markov
chain on $\{0,1,2,\ldots, N\}$ under $\Pr_N$.
The transition probabilities are
given by
\[ p_N^s (0,0) = (1-s)^2, ~~ p_N^s (0,1)= 2s(1-s), ~~ p_N^s (0,2) = s^2, \]
and for $n \geq 1$,
\begin{align}
\label{tp}
p_N^s (n,n-2) & = (1-s)^2 F_N(n) \nonumber\\
p_N^s (n,n-1) & =  2s (1-s) F_N(n) + (1-s)^2 (1- F_N(n)) \nonumber\\
p_N^s (n,n) & = s^2 F_N (n) + 2 s (1-s) (1- F_N (n) ) \nonumber\\
p_N^s ( n, n+1) & = s^2 (1 - F_N(n) ) .\end{align}
Moreover, for $n \geq 0$,
\begin{equation}
\label{drift1}
 \Exp_N [ C^N_{t+1} (s) - C^N_t (s) \mid C^N_t(s) = n ] = 2s - (1 + F_N( n) ) \1 \{ n \neq 0 \}
 .\end{equation}
\end{lm}
\begin{proof} 
Suppose that $C^N_t(s) = n$. If $n=0$, then the two points that we select come from $(s,1)$,
and $C^N_{t+1}(s)$ is 1 or 2 according to whether 1 or 2 of the new points land in
$[0,s]$: each does so, independently, with probability $s$. This gives $p_N^s(0,m)$.

Suppose that $n \geq 1$. In this case, $X_t^{(1)} \leq s$ is always removed.
Then $C^N_{t+1}(s)$ is either (i) $n-2$; (ii) $n-1$;  or (iii) $n$ according to whether
(i) the second point selected for removal is  one of $\{ X_t^{(2)},\ldots, X_t^{(n)}\}$, and both new points
fall in $(s,1)$; (ii) the second point is one of $\{ X_t^{(2)},\ldots, X_t^{(n)}\}$, and exactly one of the two new points
falls in $(s,1)$, or the second point is one of $\{ X_t^{(n+1)},\ldots, X_t^{(N)} \}$, and both new points
fall in $(s,1)$; or (iii) the second point is one of $\{ X_t^{(n+1)},\ldots, X_t^{(N)} \}$, and both new points
fall in $[0,s]$. Thus we obtain the expressions in (\ref{tp}), noting that the probability
that one of $\{ X_t^{(2)},\ldots, X_t^{(n)}\}$ is selected as the second point for removal
is $F_N(n)$. We then obtain (\ref{drift1}) from (\ref{tp}).
\end{proof}

 \subsection{The `$N = \infty$' chain $C_t(s)$}
 \label{sec:limitchain}

Again we consider the `$N = \infty$' chain $C_t(s)$
as described in Section \ref{sec:genlimitchain}.
In the special case where (A2$'$) holds, so that $F$ is the limiting
 distribution given by (\ref{Fdef}), and $\alpha =0$, the transition
 probabilities  $p^s  (n,m) = \Pr  [ C_{t+1}(s) = m \mid C_t(s) = n ]$ are
 given for
$n =0$ by
\[ p ^s(0, 0)  = (1-s)^2, ~~ p ^s(0,1)  = 2s(1-s) , ~~ p ^s(0,2)  = s^2 ,\]
and for $n\in \N$ by
\begin{align*}
p ^s(n, n-1)  = (1-s)^2, ~~ p ^s(n,n)  = 2s(1-s) , ~~ p ^s(n,n+1)  = s^2 .
\end{align*}

 In their analysis, de Boer {\em et al.} \cite{deboer}
discuss this Markov chain, although they do not give full justification
that it can be used to describe the asymptotics of the finite-$N$ chains $C^N_t(s)$;
this is justified in a specific sense by our results from Section \ref{sec:asymptotics},
which rely on the technical tools from Section \ref{sec:limits}.
In the remainder of this section we present some basic properties of $C_t(s)$.

The Markov chain $C_t(s)$ is 
{\em almost} a nearest-neighbour random walk (or birth-and-death chain), apart from the
 fact that from $0$ we can make a jump of size $2$. However, the form of the transition probabilities
 allows us to use a trick to transform this into a nearest-neighbour process (see the proof of Lemma \ref{limitlem} below).
 We will prove the following result,
  which corrects an error in the stationary distribution proposed in \cite{deboer}.

 \begin{lm}
 \label{limitlem}
 Suppose that $\alpha =0$ and $s< 1/2$.
 Then for any $n \in \Z^+$, 
 \begin{equation}
 \label{pilim}
  \lim_{t \to \infty} \Pr[ C_t (s) = n] = \pi^s (n),
  \end{equation}
 and the stationary distribution $\pi^s$ satisfies (\ref{pi}).
 Moreover,
 \begin{align}
 \label{mulim0}
  \lim_{t \to \infty} \Exp  [ C_t (s) ]
 & = 2s + \frac{s^2}{1-2s} .
 \end{align}
 \end{lm}
 \begin{proof}
  Observe that the probability of a jump from $0$ to $2$ is the same as that from $1$ to $2$ (namely, $s^2$),
 while the probability of a jump from $0$ into the set $\{0,1\}$ is the same as that from $1$
 into $\{0,1\}$ ($1-s^2$). So we can merge $\{0,1\}$ into a single state and preserve the Markov property.
 (A formal verification of the preservation of the Markov property under this transformation
 is provided by, for example, \cite[Corollary 1]{br}.)
  This   gives  a genuine birth-and-death chain on a state-space isomorphic to $\Z^+$. Call this new state-space
  $\{ \bar 0, \bar 1, \bar 2, \ldots \}$, so that $\bar 0$ corresponds to $\{0,1\}$ and $\bar n$ for $n \geq 1$
  corresponds to $n+1$ in the original state-space. This new Markov chain has transition probabilities
  \[ q^s (\bar 0, \bar 0) = 1 -s^2, ~~ q^s(\bar 0,\bar 1) = s^2, ~~\textrm{and for } \bar n \geq 1, ~~ q^s (\bar n, \bar m ) = p^s (n+1, m+1) .\]
  This Markov chain is reversible and solving the detailed balance equations (cf e.g.\ \cite[\S I.12, pp.\ 71--76]{chung})
   we obtain the stationary distribution
  $\bar \pi^s$
  for $s < 1/2$ as
  \[ \bar \pi^s ({\bar n}) = \left( 1 - \left( \frac{s}{1-s} \right)^2 \right) \left( \frac{s}{1-s} \right)^{2 n} ,\]
  for all $n \geq 0$. To obtain $\pi^s (n)$, the stationary distribution for the original Markov chain,
  we need to disentangle the composite state $\bar 0$. We have that $\pi^s(0) + \pi^s(1) = \bar \pi^s (\bar 0)$
  and, by stationarity,
  \[ \pi^s (0) = (1-s)^2 \pi^s (0) +  (1-s)^2 \pi^s (1) .\]
  Solving these equations we obtain (\ref{pi}).
   Some algebra then yields the mean of the distribution $\pi^s$ (when $s <  1/2$), giving
  \begin{equation}
  \label{pimean1}
   \sum_{n \in \Z^+} n \pi^s (n) = 2s + \frac{s^2}{1-2s} < \infty,\end{equation}
  since $s < 1/2$.
  Hence  (\ref{mulim0}) follows from Lemma  \ref{lem:limits}(c).
  \end{proof}

 \subsection{Proofs of Theorems \ref{thm0b}, \ref{thm1}, and \ref{thm2}}
\label{sec:remains}

Now we can complete the proofs of our remaining theorems.

\begin{proof}[Proof of Theorem \ref{thm0b}.]
The $s < 1/2$ statement follows from (\ref{pimean}) and   (\ref{pimean1}).
  On the other hand, for any $s \geq 1/2$ and any $r <1/2$,
we have $C_t^N (s) \geq C_t^N (r)$, so 
\[ \lim_{N \to \infty} \lim_{t \to \infty} \Exp_N [ C_t^N (s) ] \geq \lim_{r \uparrow 1/2}
\left( 2r + \frac{r^2}{1-2r} \right) = \infty ,\]
as required.
\end{proof}

\begin{proof}[Proof of Theorem \ref{thm1}.] The $s < 1/2$ part of the theorem follows from Lemma \ref{limitlem}
and Lemma \ref{lem:limits}(b).
On the other hand, for any $s \geq  1/2$ and any $r<1/2$, 
$\Pr_N [ C^N_t (s) \leq n ] \leq \Pr_N [ C^N_t (r) \leq n]$
so that, again by Lemmas \ref{limitlem} and \ref{lem:limits}(b),
\[ \lim_{N \to \infty} \lim_{t \to \infty} \Pr_N [ C^N_t (s) \leq n ] \leq
\lim_{r \uparrow 1/2} \sum_{m=0}^n \pi^r (m) = 0 ,\]
 by (\ref{pi}).
\end{proof}

\begin{proof}[Proof of Theorem \ref{thm2}.]
Since $\Pr_N [ X_t^{(n)} \leq s ] =\Pr_N [ C^N_t(s) \geq n]$, we have
\begin{align*}
\lim_{N \to \infty} \lim_{t \to \infty} \Pr_N [ X_t^{(n)} \leq s]
= 1 - \lim_{N \to \infty} \lim_{t \to \infty} \sum_{m=0}^{n-1} \Pr_N [ C_t^N (s) = m]
 = 1 - \sum_{m=0}^{n-1} \pi^s (m) ,\end{align*}
by Theorem \ref{thm1}.
The result (\ref{theta}) now follows from (\ref{pi}) and some algebra.
\end{proof}

\section{Appendix: Markov chain limits}
\label{sec:limits}

  In relating the asymptotics of the Markov chains $C_t^N(s)$ to the `$N=\infty$' Markov chain $C_t(s)$,
  we need the following general result (Theorem \ref{mclim})
   on a form of analyticity for families of Markov chains
  that are uniformly ergodic in a certain sense
  (but not the sense used in Chapter 16 of \cite{mt}, where the
  uniformity is over all possible starting states of a single Markov chain;
  our `uniformity' is over a family of Markov chains all starting at the same point).
Theorem \ref{mclim} is  related to material in Chapters 6 and 7 of \cite{fmm},
 although
our setting is somewhat different and the proof we give uses different ideas.
Our context differs from the set-up in \cite{fmm}, most notably in that
 our state-space changes with $N$, unlike
 in \cite{fmm}. It is likely that the methods of \cite{fmm}
 could be adapted to our setting. However, it is simpler to proceed directly; we use, in part, a coupling approach.
 Recall that a subset $S$ of $\R$ is locally finite if $S \cap R$ is finite for any bounded set $R$.

 \begin{theo}
 \label{mclim}
 Fix $N_0 \in \N$.
 For each integer $N \geq N_0$ let $Y^N_t$ be an irreducible, aperiodic Markov chain under $P_N$ on $S_N$ a countable subset of $[0,\infty)$,
 where $0 \in S_N$, $S_{N} \subseteq S_{N+1}$, and $\limsup_{N \to \infty} S_N = \infty$.
 Also suppose that $Y_t$ is an irreducible, aperiodic Markov chain under $P$ on $S:= \cup_N S_N$. Suppose that $S$
 is locally finite. Write $E_N$ and $E$ for expectation under $P_N$ and $P$ respectively.
 Suppose   that there exists
$B < \infty$ such that for all  $N \geq N_0$,
\begin{equation}
\label{jumps}
 P_N [ | Y^N_{t+1} - Y^N_t | > B  ] = 0 , ~\textrm{and} ~  P  [ | Y_{t+1} - Y_t | > B   ] = 0
 .\end{equation}
  Suppose also that there exist $A_0 \in (0,\infty)$ and $\eps_0>0$ for which
\begin{align}
\label{foster}
 \sup_{N \geq N_0} \sup_{x \in S_N \cap [A_0,\infty)} E_N [ Y^N_{t+1} - Y^N_t \mid Y^N_t = x ]&\leq - \eps_0; \nonumber \\
 \sup_{x \in S \cap [ A_0, \infty ) } E [ Y_{t+1} - Y_t \mid Y_t = x ]&\leq - \eps_0 .\end{align}
 Let $q_N (x,y) := P_N [ Y^N_{t+1} = y \mid Y^N_t = x]$
and $q(x,y) := P [ Y_{t+1} = y \mid Y_t = x]$.
Suppose that
\begin{equation}
\label{problim}
 \lim_{N \to \infty} [ q_N (x,y) \1 \{ x \in S_N \} ] = q  (x,y) ,\end{equation}
for all $x, y \in S$.
Then the following hold.
 \begin{itemize}
 \item[(a)]
 The Markov chain $Y_t$ is ergodic on $S$ and, for any $N \geq N_0$, $Y^N_t$ is ergodic on $S_N$.
 Let $\tau_N := \min \{ t \in \N : Y^N_t = 0 \}$ denote
 the time of the first return to $0$ by the process $Y^N_t$;
 similarly let $\tau := \min \{ t \in \N : Y_t = 0\}$.
 There exists $\delta >0$ such that
 \begin{equation}
 \label{expmom}
 E [ \re^{\delta \tau} ] < \infty , ~\textrm{and} ~ \sup_{N \geq N_0} E_N [ \re^{\delta \tau_N} ] < \infty .\end{equation}
 \item[(b)]
 There exist stationary distributions $\nu_N$ on $S_N$ and $\nu$ on $S$ such that
  \[ \lim_{t \to \infty}
P_N [ Y^N_t = x] = \nu_N (x) , ~\textrm{and}~
\lim_{t \to \infty}
P  [ Y_t = x] = \nu (x)
  .\]
Moreover,  there exist $c >0$ and $C< \infty$ such that for all $N \geq N_0$ and all $x \in S_N$,
$\nu_N (x) \leq C \re^{-cx}$ and, for all $x \in S$,
$\nu(x) \leq C \re^{-cx}$.
\item[(c)]
For any $x \in S$, $\lim_{N \to \infty} \nu_N(x)
= \nu(x)$.
\item[(d)] Finally,
\[ \lim_{t \to \infty} E_N [ Y^N_t ] = \sum_{x \in S_N} x \nu_N (x) < \infty,~\textrm{and}~
\lim_{t \to \infty} E  [ Y _t ] = \sum_{x \in S} x \nu  (x) < \infty .\]
\end{itemize}
 \end{theo}

  Before getting into the details, we sketch the outline of the proof. The Foster-type
 condition (\ref{foster}) will enable us to conclude that the Markov chains
 have a uniform (in $N$) ergodicity property implying parts (a) and (b). We then couple $Y_t$ and $Y_t^N$ on an
 interval $[0,A]$ where $A$ is chosen large enough so that the processes
 reach $0$ before leaving $[0,A]$ with high probability. Given such an $A$, we choose $N$
 large enough so that on this finite interval (\ref{problim}) ensures that the two
 Markov chains can,  with high probability, be coupled until the time that they reach $0$.
 This strategy, which succeeds with high probability, ensures that the two processes
follow identical paths over an entire excursion; using the excursion-representation
 of the stationary distributions will yield   part (c).

 An elementary
  but important consequence of the conditions of Theorem \ref{mclim}
 is a `uniform irreducibility' property that we will use repeatedly in the
 proof;   we state this property in the following result.
 Note that  the condition (\ref{problim}) is stronger than is necessary for parts (a) and (b)
 of Theorem \ref{mclim}: in the proof of Theorem \ref{mclim} (a) and (b) below,
 we use only the uniform irreducibility property given in Lemma \ref{lem:irred}.

 \begin{lm}
 \label{lem:irred}
Under the conditions of Theorem \ref{mclim}, for any $A \in (0,\infty)$ there
 exist  $\eps_1 := \eps_1 (A)>0$, $N_1(A) \in \N$, and $n_0(A) \in \N$ such that, for all $N \geq N_1(A)$ and all
  $x,y \in S_N \cap [0,A]$ there
 exists $n := n(x,y) \leq n_0 (A)$ for which
  \begin{equation}
 \label{irred}
 P_N [ Y_{n}^N = y \mid Y_0^N = x ] \geq \eps_1 .
 \end{equation}
 \end{lm}
 \begin{proof}
 Fix $A \in (0,\infty)$. Local finiteness  implies that
 $S \cap [0,A]$ is finite. 
 Take $N$ large enough so that $S_N \cap [0,A] = S \cap [0,A]$.
 Irreducibility of $Y_t$ implies that for any $x, y \in S$, there exists
 $n(x,y) < \infty$ such that $P [ Y_{n(x,y)} = y \mid Y_0 = x] >0$.
 There are only finitely many pairs $x, y \in S \cap [0,A]$, so
for such $x,y$, in fact $P [ Y_{n(x,y)} = y \mid Y_0 = x ] \geq \eps_2(A) >0$
where $n (x,y) \leq n_0 (A)$ and $\eps_2$ and $n_0$ depend only on $A$, not on $x, y$.
Moreover, for any $x, y \in S \cap [0,A]$, there are only finitely many paths
of length at most $n_0(A)$ from $x$ to $y$. It follows that for any $x, y \in [0,A]$
we can find a sequence of states of $S$, 
$x_0 =x, x_1, \ldots, x_{n-1}, x_n = y$ with $n = n(x,y) \leq n_0 (A)$ for which
\[ P [ Y_1 = x_1, \ldots, Y_n = x_n \mid Y_0 = x_0 ] \geq \eps_3 (A) > 0. \]
 However, by (\ref{problim}) we have that, as $N \to \infty$,
 \begin{align*} P_N [ Y^N_1 = x_1, \ldots, Y^N_n = x_n \mid Y^N_0 = x_0 ]
 = q_N (x_0, x_1 ) q_N (x_1, x_2 ) \cdots q_N(x_{n-1}, x_n ) \\
 \to q  (x_0, x_1 ) q  (x_1, x_2 ) \cdots q (x_{n-1}, x_n )
 = P [ Y_1 = x_1, \ldots, Y_n = x_n \mid Y_0 = x_0 ].\end{align*}
 So for all $N$ large enough, $P_N [ Y^N_n = y \mid Y^N_0 = x] \geq \eps_3 (A) /2$, say.
 \end{proof}

 Now we move on to the proof of Theorem \ref{mclim}.
 The theorem and its proof are in parts closely related to existing results in the literature, in particular
 certain results from \cite{ai2,ai3,aim,tweedie,hajek,mp,lamp3,fmm,mt} amongst others.
  However, none of the existing results that we have seen
  fits exactly into the present context, and rather than try to adapt various
 parts of these existing results we give a largely self-contained proof. In the course of the proof we give some
 more details of how the arguments relate to existing results.

 \begin{proof}[Proof of Theorem \ref{mclim} (a) and (b).]
   First we prove part (a).
  Since $Y^N_t$ is an irreducible, aperiodic Markov chain on the finite or countably infinite state-space $S_N$,
 the drift condition (\ref{foster}) enables us to apply Foster's criterion (see e.g.\ \cite{fmm}) to conclude
 that $Y^N_t$ is positive-recurrent (ergodic) and in particular, since (\ref{foster}) is uniform in $N$,
 $E_N [\tau_N]$ is uniformly bounded (independently of $N$).

 In fact, we have the much stronger result (\ref{expmom}).
 The exponential moments result (\ref{expmom}) for a specific $N$ is essentially a classical result,
 closely related to results in \cite{lamp3,fmm,mt}, for instance, and
 follows from the drift condition (\ref{foster})
 together with the bounded jumps condition (\ref{jumps}) and irreducibility: concretely, one may use, for example, Theorem 2.3 of
 \cite{hajek} or the $a =0$ case of Corollary 2 of \cite{ai2}. The uniformity in (\ref{expmom}) follows from the fact
 that (\ref{foster})
and (\ref{jumps}) hold uniformly in $N$, and that the irreducibility is also uniform in the sense of Lemma \ref{lem:irred}. Indeed, the
results of \cite{hajek,ai2} apply not to $\tau_N$ itself but to $\sigma_N := \min \{ t \in \Z^+ : Y^N_t \leq A_0 \}$
where $A_0$ is the constant in (\ref{foster}): standard arguments using the uniform irreducibility condition extend the uniform bound
on $E_N [ \re^{\delta \sigma_N} ]$ to the desired uniform bound on $E_N [ \re^{\delta \tau_N} ]$.
 In particular, (\ref{expmom}) implies that
 for any $k \in \N$ there exists $C_k < \infty$ such that
 \begin{equation}
 \label{times}
 E [ \tau^k ] \leq C_k, ~\textrm{and} ~ \sup_{N \geq N_0 } E_N [ \tau^k_N ] \leq C_k ,\end{equation}
 a fact that we will need later.  This completes the proof of part (a).

 By positive-recurrence,
 there exist    (unique) stationary distributions $\nu_N$ on $S_N$ and $\nu$ on $S$
  such that $\lim_{t \to \infty} P_N [ Y^N_t =x ] = \nu_N (x)$
  and $\lim_{t \to \infty} P [ Y_t = x] =\nu (x)$.
  Next we prove the uniform exponential decay of $\nu_N$ and $\nu$.
 Again these results are closely related to existing results in the literature,
 such as those in \cite{hajek,mp,tweedie}, Chapters 6 and 7 of \cite{fmm}, Section 2.2 of \cite{ai3}, or Section 16.3 of \cite{mt}.

  For $\delta \in (0,1)$,
 let $W_t := \re^{\delta Y_t}$. We show
 that $W_t$ has negative drift
 outside  a finite interval, provided $\delta >0$ is small enough.
  We have that
 \[ E [ W_{t+1} - W_t \mid Y_t = x] = \re^{\delta x} E [ \re^{\delta (Y_{t+1} - Y_t ) } - 1 \mid Y_t = x] .\]
 Taylor's theorem with Lagrange remainder implies that for all $y \in [-B,B]$ and all $\delta \in (0,1)$,
 $\re^{\delta y} - 1 \leq \delta y + K \delta^2$, where $K := K(B) < \infty$. Using
 this inequality and the bounded jumps assumption (\ref{jumps}), we obtain
 \begin{align*} E [ W_{t+1} - W_t \mid Y_t = x] & \leq \delta \re^{\delta x} \left( E [ Y_{t+1} - Y_t \mid Y_t =x ]
 + K \delta \right) \\
 & \leq \delta \re^{\delta x} \left( - \eps_0 + K \delta \right) ,\end{align*}
 when $x > A_0$, by (\ref{foster}). Hence, for $\delta := \delta (B, \eps_0) \in (0,1)$
 sufficiently small, we have that
 \begin{equation}
 \label{superm}
   E [ W_{t+1} - W_t \mid Y_t = x] < 0 ,\end{equation}
 for $x > A_0$.

 Let $\sigma := \min \{ t \in \Z^+ : Y_t \leq A_0 \}$ and
 $\nu_x := \min \{ t \in \Z^+ : Y_t \geq x \}$.
 By irreducibility, $\sigma \wedge \nu_x < \infty$ a.s..
 Moreover, by (\ref{superm}), $W_{t \wedge \sigma \wedge \nu_x}$ is a nonnegative supermartingale.
 Hence $W_{t \wedge \sigma \wedge \nu_x } \to W_{\sigma \wedge \nu_x}$ a.s.\ as $t\to \infty$, and
 \[ \re^{\delta Y_0} = 
 W_0 \geq E [ W_{\sigma \wedge \nu_x } ] \geq
 E [ W_{\nu_x} \1 \{ \sigma > \nu_x \} ]
 \geq \re^{\delta x} P [ \sigma > \nu_x ] .\]
 The same argument holds for $W^N_t:= \re^{\delta Y_t^N}$, uniformly in $N \geq N_0$.
 Thus we have
 \begin{equation}
 \label{hitting}
 P [ \nu_x < \sigma \mid Y_0 =y ] \leq \re^{-\delta(x-y) } , ~\textrm{and}~
 P_N [ \nu_{N,x} < \sigma_N \mid Y_0^N = y] \leq \re^{-\delta(x-y) } ,\end{equation}
where $\nu_{N,x} := \min \{ t \in \Z^+ : Y_t^N \geq x\}$ and $\sigma_N := \min \{ t \in \Z^+ : Y_t^N \leq A_0 \}$.

We deduce from (\ref{hitting}), with the uniform irreducibility property described in Lemma \ref{lem:irred},
 that the probability of reaching $[x,\infty)$ before returning to $0$ decays
 exponentially in $x$, uniformly in $N$.  
 We will show that
 \begin{equation}
 \label{return}
 P [ \nu_x < \tau ] \leq C \re^{-\delta x}, \textrm{ and } P_N [ \nu_{N,x} < \tau_N ] \leq C \re^{-\delta x} .\end{equation}

 By uniform irreducibility (Lemma \ref{lem:irred})
 and the bounded jumps assumption (\ref{jumps}), we have that there exist $A_1 \in (A_0, \infty)$
 and $\theta > 0$
 for which
 \[ \inf_{y \in S \cap [0,A_0]} P [ \tau < \nu_{A_1}  \mid Y_0 = y ] > \theta, ~\textrm{and}~
 \min_{N \geq N_0} \inf_{y \in S_N \cap [0,A_0] } P_N [ \tau_N < \nu_{N,A_1} \mid Y^N_0 = y ] > \theta .\]
 Together with (\ref{hitting}), this will yield the result (\ref{return}):
 the idea is that each time the process enters $[0,A_0]$, it has uniformly positive
 probability of reaching $0$ before it exits $[0,A_1]$, otherwise, by (\ref{hitting}),
 starting from $[A_1, A_1+B]$ the process reaches $[x,\infty)$ before its next return to $[0,A]$ with an exponentially
 small probability, and (\ref{return}) follows. We write out a more formal version of this
 idea for $Y_t$ only; a similar argument holds for $Y_t^N$.

 Let $\kappa_0 := 0$ and for $n \in \Z^+$
 define iteratively the stopping times $\eta_n := \min \{ t \geq \kappa_n : Y_t > A_1 \}$
 and $\kappa_{n+1} := \min \{ t \geq \eta_n : Y_t \leq A_0 \}$.
 By successively conditioning at these times (all of which are a.s.\ finite), we have
 \begin{align*}
 P [ \nu_x < \tau   ] & \leq P [ \nu_x < \kappa_1 \mid Y_{\eta_0} ]
 + E \left[ P [ \nu_x < \tau \mid Y_{\kappa_1} ] \mid Y_{\eta_0} \right] \\
 & \leq P [ \nu_x < \kappa_{n+1} \mid Y_{\eta_0} ]
 + E \left[ P [ \eta_1 < \tau \mid Y_{\kappa_1} ] E [ P [ \nu_x < \tau \mid Y_{\eta_1} ]
 \mid Y_{\kappa_1} ] \mid Y_{\eta_0} \right] \\
& \leq C \re^{-\delta x} \left( 1 + (1-\theta) + (1-\theta)^2 + \cdots \right) ,
 \end{align*}
 since $P[ \eta_{n} < \tau \mid Y_{\kappa_n} ] \leq 1-\theta$ a.s.,
 and $P [ \nu_x < \kappa_{n+1} \mid Y_{\eta_n} ] \leq C \re^{-\delta x}$
 by (\ref{hitting}) and the fact that $Y_{\eta_n} \leq A_1 +B$ a.s., by (\ref{jumps}). Thus we
 verify (\ref{return}).
 
 Let $L_N(x)$ denote the total occupation time of
 state $x \in S_N$ by $Y^N_t$ before time $\tau_N$, i.e., during the first excursion of $Y^N_t$;
 similarly for $L(x)$ with respect to $Y_t$. That is,
 \[ L_N (x) := \sum_{t = 0}^{\tau_N-1} \1 \{ Y^N_t = x \}, ~\textrm{and}~
 L  (x) := \sum_{t = 0}^{\tau -1} \1 \{ Y _t = x \}  .\]
 Standard theory for irreducible, positive-recurrent Markov chains (see e.g.\ \cite[Corollary I.3.6, p.\ 14]{asmussen})
 gives
 \begin{equation}
 \label{ratio} \nu_N (x) = \frac{E_N [ L_N (x)]}{E_N [ \tau_N] }, ~\textrm{and}~
 \nu  (x) = \frac{E [ L  (x)]}{E  [ \tau] } .\end{equation}
 By (\ref{return}), the probability of visiting $x$ during a single excursion
decays exponentially. In order to bound the expected occupation time, we need an estimate
for the probability of returning to $x$ starting from $x$. We claim that there exists $\eps_2 >0$ for which
\begin{equation}
\label{visitx}
\max_{N \geq N_0}
\max_{x \in S_N \cap [A_0,\infty)} P_N [ \textrm{return to } x \textrm{ before hitting } 0 \mid Y^N_0 = x ] \leq 1 -\eps_2 ,\end{equation}
and also $\max_{x \in S \cap [A_0, \infty)} P  [ \textrm{return to } x \textrm{ before hitting } 0 \mid Y_0 = x ] \leq 1 -\eps_2$.
To verify (\ref{visitx}), note that from (\ref{foster}) and (\ref{jumps})
there exists $\eps' >0$ such that $P [ Y_{t+1} - Y_t \leq -\eps' \mid Y_t = x] \geq \eps'$
for all $x \geq A_0$, and the same for $Y_t^N$ (uniformly in $N$). Then (\ref{hitting}) yields
(\ref{visitx}). It follows from (\ref{visitx}) that, starting from $x$, the number of returns (before hitting $0$)
of $Y_t$ or $Y_t^N$ to $x$ is stochastically dominated (uniformly in $N$ and $x$) by an exponential random variable.
In particular, 
\[ E_N [ L_N (x) ] \leq C P_N [ \nu_{N,x} < \tau_N ], ~\textrm{and} ~
E[ L (x) ] \leq C P [ \nu_x < \tau ] ,\]
which, with (\ref{ratio}), yields the claimed tail bounds on $\nu_N$ and $\nu$.
 \end{proof}

 \begin{proof}[Proof of Theorem \ref{mclim} (c) and (d).]
 First we prove part (c). 
 We will again use the representation (\ref{ratio}).
 We use a coupling argument to show that, as $N \to \infty$, for any $x \in S$,
 \begin{equation}
 \label{coup}
 E_N [ L_N (x)] \to E [ L (x)], \textrm{ and }
E_N [ \tau_N] \to E [ \tau ] .\end{equation}
 Let $\eps>0$. Take $A \in (0,\infty)$ large enough so that $B C_1/A < \eps$, where $C_1$ is the constant
 in the $k=1$ version of (\ref{times}) and  $B$ is the bound in (\ref{jumps}). Also, for convenience, choose $A$ so that $A/B$ is an integer.
 We claim that
 \begin{equation}
 \label{eq4}
 \lim_{N \to \infty} \sup_{x \in S_N \cap [0,2A]} \sum_{y \in S} | q_N(x,y) \1 \{ y \in S_N \} - q (x,y) | = 0 .\end{equation}
 To see this, note that since
 $S$ (and hence also $S_N \subseteq S$)
 is locally finite, in the supremum in (\ref{eq4}), $x$ takes only finitely many values (uniformly in $N$), and, by (\ref{jumps}),
 only finitely many terms in the sum are non-zero (again, uniformly in $N$). Hence by
 condition (\ref{problim}) we verify the claim (\ref{eq4}).
 By (\ref{eq4}), we can choose $N$ large enough such that
  \begin{equation}
 \label{eq3} \sup_{x \in S_N \cap[0,2A]} \sum_{y \in S} | q_N (x,y) \1 \{ y \in S_N\} - q_\infty (x, y) |  \leq \eps B /A .\end{equation}

 We couple $Y^N_t$ and $Y_t$.
 We take $N$ large enough so that $S_N \cap [0,2A+B] = S \cap [0,2A+B]$.
 We use the notation $P_N^*$ for the
 probability measure on the space on which we are going to construct coupled instances
 of $Y^N_t$ and $Y_t$, and write $E_N^*$ for the corresponding expectation.
 We start at $Y^N_0 = Y_0 \leq A$.
We claim that one can construct $(Y_t^N, Y_t)$ as a Markov chain under $P_N^*$
so that
 $P^*_N [ Y_t^N = y \mid Y_t^N =x ] = q_N (x,y)$,
 $P^*_N [ Y_t = y \mid Y_t =x ] = q (x,y)$,
and
 \[ P^*_N [ (Y_t^N, Y_t ) = (y,y) \mid ( Y_t^N, Y_t  ) = (x,x) ] \geq \min \{ q_N (x,y), q (x,y ) \} .\]
 To see this, observe that when $Y_t^N = Y_t = x$ we may
 choose transition
 probabilities $r_x (y,z) = P_N^* [ (Y_t^N,Y_t) = (y,z) \mid (Y_t^N,Y_t) = (x,x) ]$
 satisfying $r_x (y,y) = \min \{ q_N (x,y), q (x,y) \}$, $\sum_{z \neq y} r_x (y, z) = q_N (x,y)$,
 and $\sum_{y \neq z} r_x (y, z) = q  (x,z)$: these constraints can always be satisfied
 by some choice of $r_x$. If $Y_t^N \neq Y_t$, we define $P_N^*$ by allowing the two
 processes to evolve independently.

 Let $\sigma := \min \{ t \in \N : Y_t^N \neq Y_t^\infty \}$ denote the time at which the processes
 first separate.
 For any $t \leq A/B$, we have from (\ref{jumps}) that
 $\max \{ Y_t^N, Y_t \} \leq 2A$ a.s., and
 together with the fact that the state-spaces of the two processes
 coincide on $[0,2A+B]$,
  (\ref{eq3})  implies that, for $t \leq A/B$,
$P^*_N [ \sigma > t+1 \mid \sigma > t ] \geq 1 - (B \eps / A)$.
Hence, for any $t \leq A/B$,
\begin{equation}
\label{fail}
P^*_N [ \sigma > t ] \geq 1 - \left( \frac{t B \eps}{A} \right) .\end{equation}

 Let $E_t$ denote the event $E_t := \{ \sigma > t \} \cap \{ \tau_N \leq t\}$,
 i.e., that the paths of $Y_t$ and $Y_t^N$
 coincide up until time $t$ and
 visit zero by time $t$.
  Then, by (\ref{fail}),
 \begin{equation}
 \label{fail2} P^*_N [ E_{A/B}^{\rm c} ] \leq P^*_N [ \sigma \leq A/B ] + P^*_N [ \tau_N > A/B ]
 \leq \eps  + ( B C_1/A) \leq 2 \eps ,\end{equation}
 using Markov's inequality and (\ref{times}) to bound $P^*_N [ \tau_N > A/B ]$,
 and the choice of $A$ to obtain the final inequality.
  On $E_t$, $\{ \tau_N = \tau\}$, so that
 \begin{align*} E^*_N [ | \tau_N - \tau | ] & \leq E^*_N [ | \tau_N - \tau| \1 ( E_{A/B}^{\rm c} ) ] \\
 & \leq ( E^*_N [ \tau_N ^2] + E^*_N [ \tau^2 ] )^{1/2} ( P^*_N [ E_{A/B}^{\rm c} ] )^{1/2} ,\end{align*}
 by the Cauchy--Schwarz inequality.
 By (\ref{fail2}) and the $k=2$ case of (\ref{times}), this last expression is bounded above by $\eps^{1/2}$ times
 a constant not depending on $N$. Since $\eps>0$ was arbitrary, the second statement in (\ref{coup}) follows.

  Similarly, on $E_t$, $\{ L_N (x) = L (x) \}$ for any $x \in S$, so that
  \begin{align*} E^*_N [ | L_N (x) - L (x) | ] & \leq E^*_N [ | L_N (x) - L (x) | \1 ( E_{A/B}^{\rm c} ) ] \\
 & \leq ( E^*_N [ \tau_N ^2] + E^*_N [ \tau^2 ] )^{1/2} ( P^*_N [ E_{A/B}^{\rm c} ] )^{1/2} ,\end{align*}
 since $L_N (x) \leq \tau_N$ and $L(x) \leq \tau$ a.s.. Thus we obtain the first statement in (\ref{coup}).
 Combining the two statements in (\ref{coup}) with the representation in (\ref{ratio}) we obtain
 $\nu_N(x) \to \nu (x)$ for any $x \in S$, completing the proof of part (c).

 Finally we prove part (d). The convergence results follow from, for example, Theorem 2 of \cite{tweedie}
 once the integrability of the stationary distributions is established. But the fact that
 $\sum x \nu_N (x)$ and $\sum x \nu (x)$ are finite follows from the bounds in part (b).
 \end{proof}

\section*{Acknowledgements}

 AW is grateful to Edward Crane for suggesting
 the partial-order-driven model described in Section \ref{open}.

\end{document}